\documentclass[10pt,a4paper,twoside,reqno]{amsart}

\usepackage[a4paper,
left=1in,right=1in,top=1in,bottom=1.2in,%
footskip=.25in]{geometry}

\usepackage{upgreek}
\usepackage{bm}
\usepackage{latexsym, amsmath, amstext, amssymb, amsfonts, amscd, bm, array, multirow, amsbsy, mathrsfs}
\usepackage{amsthm}
\usepackage{t1enc}
\usepackage[mathscr]{eucal}
\usepackage{indentfirst}
\usepackage{pb-diagram}
\usepackage{graphicx}
\usepackage{fancyhdr}
\usepackage{fancybox}
\usepackage{enumerate}
\usepackage{color}
\usepackage{tikz-cd}
\usepackage[all]{xy}
\usepackage{hyperref}
\usepackage{tikz}
\usepackage{xparse}
\hypersetup{colorlinks=false,pdfborderstyle={/S/U/W 0}}
\usetikzlibrary{matrix}

\newcommand{\arxiv}[1]{{\tt
		\href{http://www.arXiv.org/abs/#1}{arXiv:#1}}}

\usepackage{url}
\usepackage[sort&compress,comma]{natbib}
\bibpunct{[}{]}{,}{n}{}{,}
\theoremstyle{plain}
\newtheorem{thm}{Theorem}[section]

\newtheorem{prop}[thm]{Proposition}

\newtheorem{lemma}[thm]{Lemma}
\newtheorem{cor}[thm]{Corollary}

\theoremstyle{definition}
\newtheorem{definition}[thm]{Definition}

\theoremstyle{remark}
\newtheorem{remark}[thm]{Remark}

\newtheorem*{ack}{Acknowledgements}

\def\ad{\mathrm{ad}}

\newcommand{\bDelta}{{\boldsymbol{\Delta}}}
\newcommand{\bXi}{{\boldsymbol{\Xi}}}
\newcommand{\bPhi}{\boldsymbol{\Phi}}

\newcommand{\End}{\mathrm{End}}
\newcommand{\Aut}{\mathrm{Aut}}

\newcommand{\eqdef}{\stackrel{{\rm def.}}{=}}

\newcommand{\frt}{\mathfrak{t}}
\def\fS{\mathfrak{S}}

\DeclareFontFamily{U}{rsf}{}
\DeclareFontShape{U}{rsf}{m}{n}{<5> <6> rsfs5 <7> <8> <9> rsfs7 <10-> rsfs10}{}
\DeclareMathAlphabet\Scr{U}{rsf}{m}{n}

\def\Z{\mathbb{Z}}

\def\R{\mathbb{R}}

\def\dd{\mathrm{d}}

\def\l\Xi{\overrightarrow{\Xi}}
\def\r\Xi{\overleftarrow{\Xi}}

\def\l{\partial^l}

\def\ad{\mathrm{ad}}

\def\pr{\mathrm{pr}}

\def\Diff{\mathrm{Diff}}
\def\Conf{\mathrm{Conf}}
\def\Sol{\mathrm{Sol}}
\def\Aff{\mathrm{Aff}}
\def\aff{\mathrm{aff}}

\def\Conn{\mathrm{Conn}}
\def\bzeta{\boldsymbol{\upzeta}}
\def\cConf{\mathfrak{Conf}}
\def\cSol{\mathfrak{Sol}}

\newcommand{\be}{\begin{equation*}}
\newcommand{\ee}{\end{equation*}}
\newcommand{\ben}{\begin{equation}}
\newcommand{\een}{\end{equation}}
\newcommand{\beqa}{\begin{eqnarray*}}
	\newcommand{\eeqa}{\end{eqnarray*}}
\newcommand{\beqan}{\begin{eqnarray}}
\newcommand{\eeqan}{\end{eqnarray}}

\newcommand{\id}{\mathrm{id}}

\def\cC{{\mathcal C}}

\def\cB{\Scr B}

\def\cH{\mathcal{H}}

\def\U{\mathrm{U}}

\def\cD{\mathcal{D}}
\def\cJ{\mathcal{J}}

\def\cA{\mathcal{A}}

\def\cP{\mathcal{P}}

\def\cG{\mathcal{G}}
\def\cT{\mathcal{T}}
\def\cF{\mathcal{F}}
\def\cC{\mathcal{C}}

\def\Sp{\mathrm{Sp}}
\def\G_2{\mathrm{G_2}}

\def\cO{\mathcal{O}}
\def\cL{\mathcal{L}}
\def\cS{\mathcal{S}}

\def\cV{\mathcal{V}}

\def\cX{\mathcal{X}}

\newcommand{\Hom}{{\rm Hom}}

\newcommand{\Iso}{{\rm Iso}}

\def\Aut{\mathrm{Aut}}

\def\frV{\mathfrak{V}}

\def\G{\mathrm{G}}

\def\R{\mathbb{R}}

\def\cL{\mathcal{L}}
\def\cH{\mathcal{H}}
\def\dd{\mathrm{d}}

\def\frad{\mathfrak{ad}}
\def\cM{\mathcal{M}}
\def\mC{\mathrm{C}}
\def\cl{\mathrm{cl}}
\def\dcl{{\dd_{\cD^s}\!\mbox{-}\cl}}
\def\Div{\mathrm{Div}}

\def\Lor{\mathrm{Lor}}

\usepackage{enumitem}
\setlist[itemize]{leftmargin=*}

\newcolumntype{P}[1]{>{\centering\arraybackslash}p{#1}}

\begin{document}

\title[The geometry of four-dimensional supergravity]{The geometry and DSZ quantization four-dimensional supergravity}

 \author[C. Lazaroiu]{C. Lazaroiu} \address{Department of Theoretical
	Physics, Horia Hulubei National Institute for Physics and Nuclear
	Engineering, Bucharest-Magurele, Romania}
\email{lcalin@theory.nipne.ro}

\author[C. S. Shahbazi]{C. S. Shahbazi} 
\address{Departamento de Matem\'aticas, UNED Madrid, Reino de Espa\~na}
\email{cshahbazi@mat.uned.es}
\address{Fakult\"at Mathematik, Hamburg Universit\"at, Bundesrepublik Deutschland}
\email{carlos.shahbazi@uni-hamburg.de}


\begin{abstract}
We implement the Dirac-Schwinger-Zwanziger integrality condition on four-dimensional  
classical ungauged supergravity and use it to obtain its duality-covariant, 
gauge-theoretic, differential-geometric model on an oriented four-manifold $M$ of 
arbitrary topology. Classical bosonic supergravity is completely determined by a 
submersion $\pi$ over $M$ equipped with a complete Ehresmann connection, a 
vertical euclidean metric and a vertically-polarized flat symplectic vector bundle 
$\Xi$. Building on these structures, we implement the Dirac-Schwinger-Zwanziger integrality  
condition through the choice of an element in the degree-two sheaf cohomology group with 
coefficients in a locally constant sheaf $\cL\subset \Xi$  valued in the groupoid of integral 
symplectic spaces. We show that this data determines a Siegel principal bundle $P_{\frt}$ of fixed type 
$\frt\in \mathbb{Z}^{n_v}$ whose connections provide the global geometric description of 
the local electromagnetic gauge potentials of the theory. Furthermore, we prove that 
the Maxwell gauge equations of the theory reduce to the polarized  self-duality 
condition determined by $\Xi$ on the connections of $P_{\frt}$. In addition, we investigate the 
continuous and discrete U-duality groups of the theory, characterizing them through short 
exact sequences and realizing the latter through the gauge group of $P_{\frt}$ acting on its 
adjoint bundle. This elucidates the geometric origin of U-duality, which we 
explore in  several examples, illustrating its dependence on the topology 
of the fiber bundles $\pi$ and $P_{\frt}$ as well as on the isomorphism 
type of $\cL$.
\end{abstract}
 
\maketitle

\setcounter{tocdepth}{1} 


\section{Introduction}


The main goal of this article is to construct the gauge-theoretic and 
duality-covariant global differential geometric model of the universal bosonic 
sector of four-dimensional ungauged  supergravity and study its U-duality 
group. In order to do so, we will implement the Dirac-Schwinger-Zwanziger (DSZ) 
integrality condition \cite{Dirac:1931kp,Schwinger:1966nj,Zwanziger:1968rs} 
on the gauge sector of the classical theory and we will interpret the result 
geometrically. This procedure is conceptually analogous to the implementation of
the Dirac integrality condition on the \emph{field strengths} of classical 
Maxwell theory in order to identify the notion of \emph{gauge field} with the 
notion of connection on a principal $\U(1)$ bundle, only that implemented in 
a remarkably more complex theory that requires more sophisticated sheaf cohomology 
groups. Although this type of schemes are usually referred to as \emph{DSZ quantization
conditions} in the literature, we generally prefer the term \emph{DSZ integrality
condition} since \emph{quantization} has a very different meaning in 
of quantum field theory. In the following we will use both terms interchangeably.

The {\em local} formulation of classical four-dimensional supergravity
theories has been studied intensively in the physics literature, see for
instance the seminal references \cite{Andrianopoli:1996cm,Andrianopoli:1996ve,Ceresole:1995ca,Ceresole:1995jg,Cremmer:1982en,Cremmer:1982wb,deWit:1984rvr,deWit:1984wbb} and the reviews and books \cite{Aschieri:2008ns,FreedmanProeyen,Gallerati:2016oyo,LopesCardoso:2019mlj,Ortin}. Such theories share a {\em universal bosonic sector}, which is subject 
to increasingly stringent constraints according to the number of
supersymmetry generators of the theory. The global classical  geometric 
formulation of the universal bosonic sector of supergravity was obtained in 
\cite{Lazaroiu:2016iav,Lazaroiu:2017qyr}, see also \cite{Liu:2020kds} for a 
mathematically rigorous approach to supergravity based on supergeometry. 
In \cite{Lazaroiu:2016spz,Lazaroiu:2016iav,Lazaroiu:2017qyr} it was found that 
the local structure of supergravity does not suffice to determine the theory on 
spacetimes which are not simply-connected. As show in loc. cit., the global 
formulation of the classical universal bosonic sector of ungauged 
supergravity on an oriented four-manifold  $M$ is a {\em Generalized 
Einstein-Section-Maxwell theory}, which is determined by the following data:
\begin{itemize}
\item A \emph{flat scalar bundle} $(\pi,\cH,\cG)$, where $\pi \colon X\to
M$ is a submersion equipped with a complete flat Ehresmann connection
$\cH$ and a Euclidean vertical metric $\cG$, i.e. a Euclidean metric
on the vertical bundle of $\pi$ which is invariant under the parallel
transport of $\cH$.
\item A \emph{duality bundle} $\Delta=(\cS,\omega,\cD)$ i.e. a flat
symplectic vector bundle over the total space $X$ of $\pi$.
\item A vertical polarization $\cJ$ on $\Delta$, i.e. a taming of
$(\cS,\omega)$ which is invariant under the extended parallel transport 
induced by $\cH$ and $\cD$.
\end{itemize}
The configuration space of the bosonic supergravity theory determined by 
a tuple $(\pi,\cH,\cG,\Delta,\cJ)$ as introduced above is the set of 
triples $(g,s,\cF)$, where $g$ is a Lorentzian metric on $M$, $s$ is a global 
section of $\pi$ and $\cF\in \Omega^2(M, \cS^s)$ is a two-form on $M$ taking 
values in the pull-back $\cS^s$ of $\cS$ by $s$ which is covariantly closed with 
respect to the pullback connection $\cD^s$. The classical equations of motion of the 
theory were given naturally in terms of aforementioned geometric structures as explained
in \cite{Lazaroiu:2017qyr}. In particular, the Maxwell equations (i.e. the equations 
of motion for $\cF$) correspond to the polarized self-duality condition determined 
by the taming $\cJ$.

The description of the gauge sector in terms of field strengths $\cF$ is
unsatisfactory when coupling the theory to quantized charged
particles. Indeed, the Aharonov-Bohm effect \cite{Aharonov:1959fk}
implies that this sector should admit a global description in terms of
{\em gauge potentials}, which are expected to be modeled by connections $\cA$ 
on an appropriate principal bundle $P$ defined on $M$. To determine this
bundle, we implement the Dirac-Schwinger-Zwanziger (DSZ) integrality condition 
on the gauge sector. We then show that this condition implies that $P$ is a 
Siegel bundle in the sense of \cite{LazaroiuShahbaziAGT}, i.e. a principal 
bundle whose structure group is the automorphism group of an integral 
symplectic affine torus. The latter is isomorphic to a certain semidirect product 
of an even-dimensional torus group with a modified Siegel modular group. 
This process parallels the DSZ quantization of Abelian gauge theories
with manifest electromagnetic duality, developed in \cite{LazaroiuShahbaziAGT},  
which depends on a {\em Siegel system} $Z$ on $X$. The latter was defined 
in \cite{LazaroiuShahbaziAGT} as a local system of finitely-generated free 
Abelian groups whose structure group reduces to a modified 
Siegel modular group and which is isomorphic to $\Delta$ upon tensorization 
with $\R$ over $\Z$. We define a classical configuration $(g,s,\cF)$ to be 
{\em integral} if the cohomology class of $\cF$ with respect to the de Rham 
differential twisted by $\cD^s$ belongs to the integral lattice defined by the 
second cohomology group of $M$ with coefficients in $Z^s$, the pull-back of $Z$ 
by $s$. By the results of \cite{LazaroiuShahbaziAGT}, any element of $H^2(M,Z^s)$ 
is the {\em twisted Chern class} of a Siegel bundle defined on the total space of 
$\pi$. Using this fact, we show that the DSZ quantization of classical bosonic 
supergravity is determined by the following data:
\begin{itemize}
\item A \emph{flat scalar bundle} $(\pi,\cH,\cG)$.
\item A Siegel bundle $P_{\frt}$ of type $\frt$ defined on the total space 
$X$ of $\pi$.
\item A vertical polarization $\cJ$ on the adjoint bundle $\ad(P_{\frt})$ of 
$P_{\frt}$.
\end{itemize}

\noindent
The configuration space of the DSZ quantization of bosonic supergravity 
determined by $(\pi,\cH,\cG)$ and $(P_{\frt},\cJ)$ is given by triples $(g,s,\cA)$, 
where $g$ and $s$ are as defined above and $\cA$ is a connection on $P^s_{\frt}$, 
which describes both the electric and magnetic potentials of the theory. The
Maxwell equations become a first-order condition on the connections of
$P_{\frt}$ which depends on the Hodge operator of $g$ and the polarization
$\cJ$. In particular, this gives the global and duality-covariant equations
of motion of the universal bosonic sector of four-dimensional ungauged
supergravity on $(P_{\frt},\cJ)$ in terms of the variables $(g,s,\cA)$.

Using this geometric and gauge-theoretic formulation of bosonic supergravity, 
we study its group of continuous and discrete U-duality transformations, which 
we characterize through short exact sequences involving the group of automorphisms 
of $P_{\frt}$ and its adjoint bundle. In general, these groups can differ markedly 
from their local counterparts considered in the physics literature \cite{Hull:1994ys}. 
In this regard, we emphasize the dependence of U-duality groups on the 
\emph{type}\footnote{This is a classical notion in the theory of symplectic 
lattices see for instance \cite{Debarre} or \cite[Appendix B]{LazaroiuShahbaziAGT} 
for more details.} of the corresponding Siegel modular group and of the
isomorphism class of the Siegel system $Z$, a point which does not
appear to have been noticed in the supergravity literature. In particular,
the explicit computation of the discrete U-duality groups becomes a hard
arithmetic problem in the theory of automorphisms of local systems. As an
application of the framework that we develop, we show that the group of discrete 
electromagnetic duality transformations of a four-dimensional supergravity is 
the \emph{discrete remnant} of the unbased group of automorphisms of $P_{\frt}$. 
This clarifies the geometric origin of U-duality in supergravity in terms of a 
particular class of \emph{gauge transformatons} of $\cP_{\frt}$.

The geometric formulation described in this paper provides the basis of the
mathematical framework necessary to investigate the differential-geometric 
problems arising in four-dimensional supergravity. In particular, it allows for 
a mathematically rigorous formulation of the geometric constraints imposed by 
supersymmetry through the corresponding Killing spinor equations, in the spirit 
of \cite{Cortes:2018lan}, thus opening the way for developing the mathematical 
theory of supergravity supersymmetric solutions and moduli spaces of such.


\begin{ack}
We thank Vicente Cort\'es and Tom\'as Ort\'in for useful comments and discussions. The work of C. I. L. was supported by grant IBS-R003-S1. The work of C.S.S. is supported by the Germany Excellence Strategy \emph{Quantum Universe} - 390833306 and the 2022 Leonardo Grant for Researchers and Cultural Creators, BBVA Foundation.
\end{ack}



\section{Classical bosonic supergravity}
\label{sec:geometricsugra}


In this section we recall the construction of generalized Einstein-Section-Maxwell 
theories on an oriented four-manifold $M$ given in
\cite{Lazaroiu:2016iav,Lazaroiu:2017qyr}. These give the global
differential-geometric model of the universal bosonic sector of
four-dimensional supergravity (also called \emph{classical
geometric bosonic supergravity} or classical bosonic supergravity 
for short, see \cite{Cortes:2018lan}).


\subsection{Preparations}


Let $M$ be an oriented and connected four-manifold. We start by
introducing the geometric data needed to formulate classical 
bosonic supergravity on $M$, a detailed account of which was given 
in \cite{Lazaroiu:2017qyr}.

\begin{definition} A \emph{scalar bundle} of rank $n_s$ on $M$ is a
triple $(\pi,\cH,\cG)$ consisting of:
\begin{itemize}
\item A smooth submersion $\pi\colon X\to M$, where $X$ is a
connected and oriented differentiable manifold of dimension $n_s + 4$.
\item A complete Ehresmann connection $\cH \subset TX$ on
$\pi$.
\item A {\em vertical Euclidean metric}, i.e. a Euclidean metric $\cG$
defined on the vertical bundle $\cV\subset TX$ of $\pi$ which is
preserved by the parallel transport of $\cH$. Recall that the vertical
bundle $\cV\subset TX$ is defined as the kernel of the differential
map $\dd\pi\colon TX\to TM$.
\end{itemize}

\noindent We say that a scalar bundle $(\pi,\cH,\cG)$ is \emph{flat}
if $\cH$ is Frobenius integrable.
\end{definition}

\noindent In the following we denote by $\cO_m$ the restriction of any 
geometric structure $\cO$ defined on $X$ to the fiber $X_m$ of $\pi$ at 
$m\in M$.

\begin{remark} Since the Ehresmann connection $\cH$ of a scalar bundle
is complete and its parallel transport preserves the vertical metric
$\cG$, the fibers $(X_m,\cG_m)$ of $\pi$ are isomorphic to each other
as Riemannian manifolds. By the results of \cite{Ehresmann}, it
follows that $\pi$ is a fiber bundle associated to a principal bundle
with structure group given by the isometry group of $(X_m,\cG_m)$. 
Notice that $X_m$ is typically non-compact in physics applications. 
\end{remark}

\begin{definition} A scalar bundle $(\pi,\cH,\cG)$ is:
\begin{itemize}
\item \emph{topologically trivial} if $\pi$ is topologically trivial
as fiber bundle, i.e. $X$ is diffeomorphic with $M\times \cM$ for some
manifold $\cM$ and $\pi$ identifies with the projection $\pr:M\times
\cM\rightarrow M$ on the first factor.
\item \emph{holonomy-trivial} if the holonomy of $\cH$ is trivial.
\end{itemize}
\end{definition}

\begin{remark} Every holonomy-trivial scalar bundle is topologically
trivial. Moreover, its Ehresmann connection identifies with the
pull-back of $TM$ through the projection $\pr:M\times
\cM\rightarrow M$.
\end{remark}

\noindent
Let $\cP(M)$ and $\cP(X)$ respectively be the sets of piece-wise 
smooth paths in $M$ and $X$ defined on the unit interval. Given a 
scalar bundle$(\pi,\cH,\cG)$, let $T$ be the parallel transport 
defined by $\cH$, which associates to a path $\gamma\in \cP(M)$ 
the diffeomorphism:
\begin{equation*}
T_{\gamma}\colon X_{\gamma(0)} \xrightarrow{\sim}  X_{\gamma(1)}\, ,
\end{equation*}

\noindent
obtained by parallel transport along $\cH$.

\begin{definition}
A \emph{duality bundle} $\Delta = (\cS,\omega,\cD)$ over $\pi$ is a triple
$(\cS,\omega,\cD)$ where $\cS$ is a vector bundle on $X$, $\omega$ is a symplectic 
structure on $\cS$ and $\cD$ is a flat connection on $\cS$ preserving $\omega$.
We denote the rank of $\cS$ by $2n_v$.
\end{definition}

\begin{remark} As shown in \cite{Lazaroiu:2017qyr}, a duality bundle
of rank $2n_v$ corresponds locally to a supergravity theory coupled to
$n_v$ \emph{vector multiplets}.
\end{remark}

\noindent
Let $(\pi,\cH,\cG)$ be a scalar bundle over $M$ and $\Delta =
(\cS,\omega,\cD)$ be a duality bundle on $\pi$, which we shall also
call a duality bundle over $(\pi,\cH,\cG)$. For any $m\in M$, let
$(\cS_m,\omega_m,\cD_m)$ be the restriction of $(\cS,\omega,\cD)$ to
the fiber $X_m = \pi^{-1}(m)$. This is a flat symplectic vector bundle
on $X_m$ and hence a duality structure on the latter as defined in
\cite{Lazaroiu:2016iav}. For any path $\Gamma\in\cP(X)$ in the total
space $X$ of $\pi$, we denote by
$\mathfrak{U}_\Gamma:\cS_{\Gamma(0)}\rightarrow \cS_{\Gamma(1)}$ the 
parallel transport of $\cD$ along $\Gamma$. Since $\cD$ is a
symplectic connection, $\mathfrak{U}_\Gamma$ is a symplectomorphism
between the symplectic vector spaces $(\cS_{\Gamma(0)},
\omega_{\Gamma(0)})$ and $(\cS_{\Gamma(1)},\omega_{\Gamma(1)})$. For
any $\gamma\in \cP(M)$, let $\bar{\gamma}_x\in \cP(X)$ be the
horizontal lift of $\gamma$ starting at the point $x\in
X_{\gamma(0)}$. By the definition of $T$, we have $\bar{\gamma}_x(1) =
T_\gamma(x)$.

\begin{definition}
The {\em extended horizontal transport} along a path $\gamma\in
\cP(M)$ is the unbased isomorphism of flat symplectic vector bundles
$\cT _\gamma:\cS_{\gamma(0)}\rightarrow \cS_{\gamma(1)}$ defined by:
\begin{equation*}
\cT_\gamma(x)\eqdef \mathfrak{U}_{\bar{\gamma}_x}\colon\cS_{x}\rightarrow \cS_{T_\gamma(x)}\, , \quad \forall\, x\in X_{\gamma(0)}\, ,
\end{equation*}	
which linearizes the Ehresmann transport $T_\gamma:X_{\gamma(0)}\rightarrow X_{\gamma(1)}$ along $\gamma$.
\end{definition}
 
\noindent
Given a duality bundle $\Delta=(\cS,\omega,\cD)$, a (compatible)
\emph{taming} $\cJ\in \Aut_b(\cS)$ on $\Delta$ is a complex structure 
on $\cS$ which tames the symplectic pairing $\omega$, i.e. it satisfies 
the compatibility condition:
\begin{equation*}
\omega(\cJ \xi_1, \cJ \xi_2 ) = \omega(\xi_1,\xi_2)\, , \qquad \forall\,\, (\xi_1 , \xi_2) \in \cS\times_X\cS, ,
\end{equation*}
and the positivity condition:
\begin{equation*}
\omega(\xi,\cJ\xi) > 0\, , \qquad \forall\,\, \xi\in \dot{\cS}~~,
\end{equation*}
where $\dot{\cS}$ is the complement of the image of the zero section in $\cS$.

\begin{definition}
\label{def:verticalJ}
A taming on the duality bundle $\Delta$ over $(\pi,\cH,\cG)$ is called
\emph{vertical} if it is preserved by the extended horizontal
transport $\cT$, i.e. if $\cT_\gamma\colon
(\cS_{\gamma(0)},\omega_{\gamma(0)} , \cJ_{\gamma(0)})\rightarrow
(\cS_{\gamma(1)},\omega_{\gamma(1)}, \cJ_{\gamma(1)})$ an isomorphism
of tamed symplectic vector bundles for all $\gamma\in \cP(M)$.
\end{definition}

\begin{remark} As explained in \cite{Lazaroiu:2016iav}, a vertical
taming on $\Delta$ is equivalent to a \emph{positive} Lagrangian
sub-bundle of the complexification of $(\cS,\omega)$ which is preserved by
the complexified extended horizontal transport.
\end{remark}

\noindent
As in \cite{Lazaroiu:2016iav,Lazaroiu:2017qyr}, we refer to the pair
$\Xi = (\Delta,\cJ)$ consisting of a duality bundle $\Delta$ defined
on $(\pi,\cH,\cG)$ and a vertical taming $\cJ$ as an
\emph{electromagnetic bundle} on $(\pi,\cH,\cG)$. We will refer to a
choice scalar bundle $(\pi,\cH,\cG)$ together with a choice of
electromagnetic bundle $\Xi$ as a {\em scalar-electromagnetic bundle}
$\Phi$, that is:
\begin{equation*}
\Phi \eqdef (\pi,\cH,\cG,\Xi)\, .
\end{equation*}

\noindent
As shown in \cite{Lazaroiu:2016iav,Lazaroiu:2017qyr}, the universal
bosonic sector of supergravity defined on $M$ is determined by the
choice of a scalar-electromagnetic bundle. Morphisms of duality and
electromagnetic bundles are defined in the natural way (see
\cite{Lazaroiu:2017qyr}). Note that standard bundle theory implies that 
isomorphism classes of duality bundles over a fixed submersion 
$\pi:X\rightarrow M$ are in one to one correspondence with the character 
variety:
\begin{equation*}
\mathfrak{M}_d (X) \eqdef \Hom(\pi_1(X), \mathrm{Sp}(2n_v,\mathbb{R}))/\mathrm{Sp}(2n_v,\mathbb{R}) \, .
\end{equation*} 

\begin{remark} In general, the character variety above has positive
dimension, giving a moduli space of inequivalent duality bundles. This
implies \cite{Lazaroiu:2017qyr} that one can construct an uncountable
infinity of globally inequivalent bosonic geometric supergravities
which are however all locally equivalent.
\end{remark}

\begin{remark} A duality bundle $\Delta=(\cS,\omega,\cD)$ is called
{\em topologically trivial} if the vector bundle $\cS$ is trivial,
i.e. if it admits a global frame.  It is called \emph{symplectically
trivial} if the $(\cS,\omega)\in \Delta$ is symplectically trivial,
i.e. if $\cS$ admits a global {\em symplectic} frame. Finally, we say
that $\Delta$ is \emph{holonomy trivial} if the holonomy of $\cD$ is
the trivial group. Holonomy-triviality implies symplectic triviality,
which in turn implies topological triviality.  If $X$ is simply
connected then every duality bundle is holonomy trivial.
\end{remark}

\noindent Smooth sections of the submersion $\pi:X\rightarrow M$ are
called \emph{scalar sections}. For every scalar section $s\colon
M\to X$ we use a superscript $s$ to denote the bundle pull-back
by $s$ and the subscript $s$ to denote push-forward by
$s$ in the appropriate category. For instance,
$\Delta^{s}=(\cS^s,\omega^s,\cD^s)$ denotes
the bundle pull-back of $\Delta=(\cS,\omega,\cS)$ by $s$, which
is a flat symplectic vector bundle over $M$. Similarly,
$\Xi^{s}=(\Delta^s,\cJ^s)$ denotes the pull-back of
$\Xi=(\Delta,\cJ)$ by $s$, which is an electromagnetic structure
on $M$ in the sense of \cite{Lazaroiu:2016iav}. Let $\Phi$ be a
scalar-electromagnetic bundle on $M$. For every Lorentzian metric $g$
on $M$ and every scalar section $s\in \Gamma(\pi)$, consider the
isomorphism of vector bundles:
\begin{equation*}
\star_{g, \cJ^{s}} \colon \wedge T^{\ast}M \otimes \cS^{s} \to  \wedge T^{\ast}M \otimes \cS^{s}\, ,
\end{equation*}
defined through $\star_{g,\cJ^{s}}=\ast_g\otimes\cJ^s$. Since both 
$\ast_g$ and $\cJ^{s}$ square to minus the identity, this restricts 
to an involutive automorphism:
\begin{equation*}
\star_{g,\cJ^{s}} \colon \wedge^2 T^{\ast}M \otimes \cS^{s} \to  \wedge^2 T^{\ast}M \otimes \cS^{s}
\end{equation*}
which gives a direct sum decomposition into eigenbundles corresponding 
to the eigenvalues $+1$ and $-1$:
\begin{equation*}
\wedge^2 T^{\ast}M \otimes \cS^{s} = (\wedge^2 T^{\ast}M \otimes \cS^{s})_{+}  \oplus (\wedge^2 T^{\ast}M \otimes \cS^{s})_{-}\, .
\end{equation*}
Here the subscript denotes the sign of the corresponding eigenvalue of
$\star_{g,\cJ^s}$. The spaces of smooth global sections of these
sub-bundles are denoted by $\Omega^2_\pm(M,\cS^{s})$ and their
elements are called polarized (anti)-selfdual $\cS^s$-valued
2-forms with respect to $\cJ^s$. We have:
\begin{equation*}
\Omega^2(M,\cS^{s}) = \Omega^2_{+}(M,\cS^{s}) \oplus \Omega^2_{-}(M,\cS^{s})\, ,
\end{equation*}
The flat symplectic connection $\cD^{s}$ of $\Delta^{s}$
defines an exterior covariant derivative acting on
$\cS^s$-valued forms defined on $M$, which we denote by:
\begin{equation*}
\dd_{\cD^{s}}\colon \Omega(M,\cS^{s}) \to \Omega(M,\cS^{s})\, .
\end{equation*}
This operator squares to zero since $\cD^{s}$ is flat. We 
denote its cohomology groups by $H^{k}(M,\Delta^{s})$ and
the corresponding total cohomology by $H(M,\Delta^{s})$. For
every scalar section $s\in \Gamma(\pi)$, we denote by 
$\mathfrak{G}^{s}_{\Delta}$ the sheaf of flat sections of 
$\Delta^s$, defined as follows:
\begin{equation*}
\mathfrak{G}^{s}_{\Delta}(U) \eqdef \left\{ \xi \in \Gamma(U,\cS^{s}) \,\, \vert\,\, \cD^{s} \xi = 0\right\}~~,
\end{equation*}
for any open set $U\subset M$.  This is a locally-constant sheaf of
symplectic vector spaces of rank $2n_{v}$, whose stalk is isomorphic
to the typical fiber of $\Delta$. Since the sheaf of smooth
$\cS^s$-valued forms is acyclic, there exists a natural
isomorphism of graded vector spaces:
\begin{equation*}
H(M,\Delta^{s}) \simeq H(M,\mathfrak{G}^{s}_{\Delta})\, , 
\end{equation*}
where $H(M,\mathfrak{G}^{s}_{\Delta})$ is the sheaf cohomology
of $\mathfrak{G}^{s}_{\Delta}$. Note that the definition of an
electromagnetic bundle $\Xi = (\Delta,\cJ)$ does not require $\cD\in
\Delta$ to be compatible with $\cJ$, a fact which is crucial for
recovering the correct local description of bosonic geometric
supergravity. The failure of $\cD$ to be compatible with $\cJ$ is
measured by the \emph{fundamental form} of an electromagnetic bundle.

\begin{definition}
Let $\Phi$ be an scalar-electromagnetic bundle. The {\em fundamental
form} $\Psi$ of $\Xi$ is the following $End(\cS)$-valued one-form 
defined on $X$:
\begin{equation*}
\Psi \eqdef \cD\cJ \in \Omega^1(X,End(\cS))\, .
\end{equation*} 
\end{definition}

\begin{remark} 
For every $v\in \Gamma(TX)$, the endomorphism
$\Psi(v)\in \End(\cS)=\Gamma(End(\cS))$ is $\cJ$-antilinear
$Q$-symmetric, where the scalar product $Q$ is the Euclidean metric
induced by $\omega$ and $\cJ$ on $\cS$ as follows:
\be
Q(\xi_1,\xi_2) \eqdef \omega(\xi_1,\cJ \xi_2)~~\forall (\xi_1,\xi_2)\in \cS\times_X \cS\, ,
\ee
See \cite{Lazaroiu:2016iav,Lazaroiu:2017qyr} for more details.
\end{remark}

\begin{definition}
An electromagnetic bundle $\Xi$ is called {\em unitary} if $\Psi = 0$.
\end{definition}

\noindent To describe the universal bosonic sector of 4d supergravity, 
we introduce three natural operations which are determined by a choice of a
scalar-electromagnetic bundle $\Phi$, a Lorentzian metric $g$ on $M$ and a 
scalar section $s \in \Gamma(\pi)$.

\begin{definition}
\label{def:tep}
The {\em twisted exterior pairing} $(\cdot,\cdot)_{g,Q^{s}}$ is
the unique pseudo-Euclidean scalar product on $\wedge
T^{\ast}M\otimes\cS^{s}$ which satisfies:
\begin{equation*}
(\rho_1\otimes \xi^{s}_1,\rho_2\otimes \xi^{s}_2)_{g,Q^{s}}=(\rho_1,\rho_2)_g Q^{s}(\xi^{s}_1,\xi^{s}_2) =(\rho_1,\rho_2)_g Q^s(\xi_1,\xi_2)
\end{equation*}
for all $\rho_1,\rho_2\in \Omega(M)$ and all $\xi_1,\xi_2\in
\Gamma(\cS^{s})$.
\end{definition}

\noindent 
Given any vector bundle $W$ on $M$, we extend this trivially to a
$W$-valued pairing (which for simplicity we denote by the same
symbol) between the bundles $W\otimes \wedge T^{\ast}M\otimes
\cS^{s}$ and $\wedge T^{\ast}M\otimes \cS^{s}$. Thus:
\begin{equation*}
(w \otimes \eta_1,\eta_2)_{g,Q^{s}} =  w \otimes (\eta_1,\eta_2)_{g,Q^{s}}\, , \quad\forall\,\, w\in \Gamma(W)\, , \quad \forall\,\, \eta_1,\eta_2\in \Omega(M,\cS^{s})~~.
\end{equation*}

\begin{definition}
The {\em inner $g$-contraction of (2,0)-tensors} is the bundle
morphism $\oslash_g:(\otimes^2T^\ast M)^{\otimes
2}\rightarrow\otimes^2 T^\ast M$ uniquely determined by the condition:
\begin{equation*}
(\alpha_1\otimes\alpha_2)\oslash_g (\alpha_3\otimes \alpha_4)=(\alpha_2,\alpha_4)_g\alpha_1\otimes \alpha_3\, , \quad \forall\, \alpha_1, \alpha_2, \alpha_3, \alpha_4\in T^{\ast}M\, .
\end{equation*}
We define the \emph{inner $g$-contraction of two-forms} to be the
restriction of $\oslash_g$ to $\wedge^2 T^\ast M \otimes \wedge^2
T^\ast M\subset (\otimes^2T^\ast M)^{\otimes 2}$.
\end{definition}

\begin{definition}
The {\em twisted inner contraction} of $\cS^{s}$-valued
two-forms is the unique morphism of vector bundles:
\begin{equation*}
\oslash_{Q^{s}}\colon\wedge^2 T^{\ast}M\otimes\cS^{s}\times_M \wedge^2 T^{\ast}M\otimes\cS^{s}\rightarrow\otimes^2(T^\ast M)
\end{equation*}
which satisfies:
\begin{equation*}
(\rho_1\otimes s_1)\oslash_{Q^{s}} (\rho_2\otimes s_2)= Q^{s} (s_1,s_2) \rho_1 \oslash_g\rho_2\, ,
\end{equation*}
for all $\rho_1,\rho_2\in \Omega^2(M)$ and all $s_1,s_2\in \Gamma(\cS^{s})$. 
\end{definition}


\subsection{The configuration space and equations of motion}


We are ready to give the geometric formulation of the universal bosonic sector of 
4d classical supergravity, whose global solutions can be interpreted locally as 
geometric classical supergravity U-folds \cite{Lazaroiu:2016spz,Lazaroiu:2017qyr}. 

\begin{definition}
Let $\Phi=(\pi,\cH,\cG,\Xi)$ be a scalar-electromagnetic
bundle on an oriented four-manifold $M$. The {\bf configuration space}
of the universal bosonic sector determined by $(M,\Phi)$ is the set:
\begin{equation*}
\Conf(\Phi) \eqdef \left\{ (g,s,\cF)  \,\, \vert \,\, g\in \Lor(M) \, , \,\, s\in \Gamma(\pi)\, , \,\, \cF  \in \Omega^2_\dcl(M,\cS^{s}) \right\}\, ,
\end{equation*}
where $\Lor(M)$ denotes the set of Lorentzian metrics on $M$ and $\Omega^2_\dcl(M,\cS^{s})$ denotes the set of $\cD^s$-closed 2-forms on $M$ valued in $\cS^s$. 
\end{definition}

\begin{remark} In general, the isomorphism class of $\cS^{s}$
depends on the scalar section $s\in \Gamma(\pi)$.
\end{remark}

\noindent
Given a scalar bundle $(\pi,\cH,\cG)$, the complete Ehresmann
connection $\cH$ can be described through a one-form $\cC\in
\Omega^1(X,\cV)=\Hom(TX,\cV)$ which restricts to the identity on $\cV\subset TX$ 
and satisfies the condition $\cC\circ\cC = \cC$. Thus $\cC\colon TX\to \cV$ 
is a projection of the tangent bundle of $X$ onto $\cV$. The horizontal 
distribution $\cH$ is recovered as the kernel of $\cC$. Given $(g,s,\cF) \in
\Conf(\Phi)$, we define the {\em vertical first fundamental form}
$(s^{\ast}_{\cC}\cG)\in \Gamma(T^{\ast}M\odot T^{\ast}M)$
through \cite{WoodI,WoodII}:
\begin{equation*}
(s^{\ast}_{\cC}\cG)(v_1 , v_2) \eqdef \cG(\cC\circ \dd s(v_1),\cC\circ \dd s(v_2))\, , \quad (v_1,v_2)\in TM\times_M TM\, ;
\end{equation*}
which depends explicitly on $\cC$ or, equivalently, $\cH$. 
The trace of this tensor with respect to $g$ is called the
\emph{vertical tension} of the section $s\in \Gamma(\pi)$. For
ease of notation, we define $\dd^{\cC} s(v) = \cC\circ \dd
s(v)\in \cV$, where $v\in TM$. Notice that $\dd^{\cC} s\in
\Omega^1(M,\cV^{s})$. Denote by $\Lor(X)$ the set of
Lorentzian metrics on $X$. Every Lorentzian metric $g$ on $M$ 
can be lifted to $\cH$ using the isomorphism of vector bundles 
$(\dd\pi\vert_{\cH})\colon \cH \xrightarrow{\sim} TM$ given by the 
restriction of $\dd\pi$ to $\cH$. Thus there exists a natural map 
(see \cite{Lazaroiu:2017qyr}):
\begin{equation*}
h\colon \Conf(\Phi) \to \Lor(X)\, , \quad (g,s,\cF)\mapsto h(g) \eqdef \pi^{\ast}g + \cG\, ,
\end{equation*}
where $h(g)$ is written using the direct sum decomposition $TX = \cH
\oplus \cV$. The Lorentzian metric $h(g)$ enters the equations of motion 
of bosonic supergravity, as explained below. Note that, equipped with the 
lifted  metric $h(g)$, $\pi \colon (X,h(g))\to (M,g)$ becomes a Lorentzian 
submersion. Given $(g,s,\cF) \in \Conf(\Phi)$, we denote by
$\nabla^{h(g)}$ the Levi-Civita connection defined by $h(g)$ on
$X$. For every scalar section $s \colon M\to X$, we denote by
$\nabla^{\Phi(g,s)}$ the connection on $TM\otimes \cV^{s}$
given by the tensor product of the Levi-Civita connection $\nabla^g$
of $g$ with the the vertical projection of the pull-back by $s$
of the connection $\nabla^{h(g)}$. We then have:
\begin{equation*}
\nabla^{\Phi(g,s)}\dd^{\cC}s \in \Gamma(T^{\ast}M\otimes T^{\ast}M\otimes \cV^{s})\, ,
\end{equation*}
as explained in \cite{WoodI,WoodII}. In particular:
\begin{eqnarray*}
\mathrm{Tr}_g(\nabla^{\Phi(g,s)}\dd^{\cC}s) \in \Gamma(\cV^{s})\, .
\end{eqnarray*}

\noindent
Given $(g,s,\cF) \in \Conf(\Phi)$, we recall that $ \Psi^s \in \Gamma((\cV^s)^{\ast}\otimes End(\cS^s))$ denotes the pull-back of $\Psi\in \Omega^1(X, End(\cS))$ by $s\colon M\to X$. Hence $(\Psi^s)^{\sharp_{\cG}}\in \Gamma(\cV^s\otimes End(\cS^s))$. For further reference we introduce symbol $(\Psi^s)^{\sharp_{\cG}}\cF_{\cA}\in \Omega^2(M,\cV^s \otimes \cS^s)$, which by definition denotes the action of $\Psi^{s}$ on $\cF_{\cA}$ as an endomorphism of $\cS^s$ while tensoring with $\cV^s$.

\begin{definition}
Let $\Phi$ be a scalar-electromagnetic bundle on $M$. The universal
bosonic sector defined by $\Phi$ on $M$ is described by following
system of partial differential equations for triples $(g,s,\cF)
\in \Conf(\Phi)$:
\begin{itemize}
\item The Einstein equations:
\begin{equation}
\label{eq:GlobalEinstein}
\mathrm{Ric}^g - \frac{g}{2} \mathrm{R}^g = \frac{1}{2} \mathrm{Tr}_g(s^{\ast}_{\cC}\cG)\, g - s^{\ast}_{\cC}\cG   + 2 \cF \oslash_{Q^{s}} \cF\, ,
\end{equation}
where $\mathrm{Ric}^g$ and $\mathrm{R}^g$ are respectively the Ricci tensor 
and Ricci scalar of $g$, while $\mathrm{Tr}_g$ denotes trace with respect to 
$g$.
\item The scalar equations:
\begin{equation}
\label{eq:GlobalScalar}
\mathrm{Tr}_g(\nabla^{\Phi(g,s)}  \dd^{\cC}s) = \frac{1}{2} (\ast \cF , (\Psi^s)^{\sharp_{\cG}}\cF)_{g,Q^{s}}\, .
\end{equation}		
\item The Maxwell equations:
\begin{equation}
\label{eq:GlobalMaxwell}
\star_{g,\cJ^s}\cF = \cF\, .
\end{equation}
\end{itemize}
We denote by $\mathrm{Sol}(\Phi) \subset \Conf(\Phi)$ the set of solutions to these equations.
\end{definition}

\begin{remark}
The configuration space $\Conf(\Phi)$ is formulated using the
\emph{field strength} two-forms instead of the appropriate notion of
gauge potential, as required by the Aharonov-Bohm effect \cite{Aharonov:1959fk}.  
The latter suggests that the gauge potentials of the theory should be described 
by connections on an appropriate principal bundle. To identify this
bundle, we must impose an appropriate DSZ integrality condition on
the field strength $\cF$. We consider this condition and its
geometric interpretation in Section \ref{sec:DQsymplecticabelian}.
\end{remark}

\noindent
The fact that the formulation given above reduces {\em
locally} to the usual formulas of local bosonic supergravity found in
the physics literature was proved in detail in references
\cite{Lazaroiu:2016iav,Lazaroiu:2017qyr}, to which we refer the reader
for further details. It is not known if this theory can be
supersymmetrized when $\cH$ is not flat, although the Killing spinor
equations can be formulated exactly as in the case when $\cH$ is flat.


\subsection{The classical U-duality group}
\label{sec:classicalglobalautgroup}


In this section we characterize the \emph{global} U-duality group of the 
bosonic supergravity associated to a fixed scalar electromagnetic bundle 
$\Phi = (\pi,\cH,\cG,\Xi)$. Given a duality bundle $\Delta = (\cS,\omega,\cD)$
let $\Aut(\cS)$ denote the group of all \emph{unbased} automorphisms
of the vector bundle $\cS$.  Let $f_u \in \Diff(X)$ be the
diffeomorphism covered by $u\in \Aut(\cS)$. Moreover, let
$\Aut(\Delta)$ be the group of those unbased automorphisms of $\cS$
which preserve both $\omega$ and $\cD$:
\begin{equation*}
\Aut(\Delta) \eqdef \left\{ u\in \Aut(\cS)\,\,\vert\,\, \omega^u = \omega\, , \,\, \cD^u = \cD \right\}\, .
\end{equation*}
Let $\Aut_{\pi}(\Delta)$ be the subgroup consisting of all elements of
$\Aut(\Delta)$ which cover based automorphisms of the fiber bundle
$\pi$:
\begin{equation*}
\Aut_{\pi}(\Delta) \eqdef\left\{ u\in \Aut(\Delta)\,\, \vert\,\, f_u\in \Aut_b(\pi)\right\}=\left\{ u\in \Aut(\Delta)\,\, \vert\,\, \pi\circ f_u=\pi\right\}\, .
\end{equation*}
We have a short exact sequence of groups:
\ben
\label{ses1}
1 \to \Aut_b(\Delta) \to \Aut_{\pi}(\Delta) \to \Aut_{b}^0(\pi)\to 1\, ,
\een
where $\Aut_{b}^0(\pi)\subset \Aut_{b}(\pi)$ is the subgroup
of those automorphisms of $\pi$ which are covered by elements of
$\Aut(\Delta)$.

Given a scalar bundle $(\pi,\cH,\cG)$ and an element $u\in
\Aut_{\pi}(\Delta)$, the fiber bundle automorphism $f_u\in
\Aut_b(\pi)$ covered by $u$ acts as a gauge transformation on $\cH$
through push-forward $\cH_{u} \eqdef (f_u)_{\ast}\cH$. Similarly, 
since $f_u$ is an automorphism of $\pi$ covering the identity,
the push-forward of $\cG$ by $f_u$ defines a new vertical Riemannian 
metric $\cG_u \eqdef (f_{u})_\ast\cG$ on $\pi$ such that $(X_m,\cG_m)$ 
is isometric to $(X_m,(\cG_u)_m)$ for all $m\in M$. Given an electromagnetic 
bundle $\Xi=(\Delta,\cJ)$, push-forward by $f_u$ produces another
electromagnetic bundle which we denote by $(\Delta_u,\cJ_u)$. Given a
scalar-electromagnetic bundle $\Phi = (\pi,\cH,\cG,\Delta,\cJ)$, the
system:
\begin{equation*}
\Phi_u \eqdef (\pi,\cH_u,\cG_u,\Delta_u,\cJ_u)
\end{equation*}
is a scalar-electromagnetic bundle with the same underlying
submersion $\pi\colon X\to M$. If $\cC\in \Omega^1(X,\cV)$ is
the connection one-form associated to $\cH$ then the natural
push-forward $f_{u\ast}\cC\in\Omega^1(X,\cV)$ is the connection 
one-form associated to $\cH_u$.

\begin{remark} Since elements of $\Aut(\cS)$ may cover non-trivial
diffeomorphisms of $X$, the pull-back or push-forward operations must
be dealt with care (see \cite{Lazaroiu:2016iav}). Explicitly, define
the following action of $\Aut(\cS)$ on sections of $\cS$:
\begin{equation*}
u\cdot \xi = u\circ \xi \circ f_u^{-1} \colon M \to \cS \, , \qquad u\in \Aut(\cS)\, , \qquad \xi\in \Gamma(\cS)\, .
\end{equation*}
This gives an isomorphism of real vector spaces $u\colon \Gamma(\cS)
\to \Gamma(\cS)$ for every element $u\in \Aut(\cS)$. We have
$\omega^u = \omega$ if and only if:
\begin{equation*}
(\omega^u)(\xi_1,\xi_2) \eqdef \omega(u\cdot \xi_1 , u\cdot \xi_2) \circ f_u = \omega(\xi_1,\xi_2)\, , \qquad \forall\,\, \xi_1 , \xi_2 \in \Gamma(\cS)\, .
\end{equation*}
Likewise, we have $\cD^u = \cD$ if and only if:
\begin{equation*}
\cD^u_{v}(\xi) \eqdef u^{-1}\cdot \cD_{f_{u\ast}\cdot v} (u\cdot \xi) = \cD_v (\xi)\, , \quad \forall\,\, \xi\in \Gamma(\cS)\, , \quad \forall\,\, v\in \Gamma(TX)\, ,
\end{equation*}
where $f_{u\ast}\cdot v = \dd f_{u}(v)\circ f^{-1}_u$ and $\dd f_u \colon TX\to TX$ 
is the ordinary differential of $f_u\in \Diff(X)$. Recall that if 
$v\in \Gamma(TX)$ then $\dd f_{u}(v)$ is not a vector field on $X$ but 
a section of $TX$ along $f_u$, whereas $f_{u\ast }\cdot v \in \Gamma(TX)$ 
is again a vector field on $X$. To illustrate the inner workings of the 
pull-backed connection $D^u$ we verify that it satisfies the Leibniz identity:
\begin{eqnarray*}
&\cD^u_{v}(\kappa\, \xi) = u^{-1}\cdot \cD_{f_{u\ast} \cdot v} (u\cdot (\kappa\, \xi)) = u^{-1}\cdot \cD_{f_{u\ast} \cdot v} ((\kappa\circ f_u^{-1})\, u\cdot \xi) = u^{-1}\cdot (\dd (\kappa\circ f_u^{-1})(f_{u\ast} \cdot v) \,u\cdot \xi) \\
& + u^{-1}\cdot (\kappa\circ f_u^{-1}\, \cD_{f_{u\ast} \cdot v}(u\cdot\xi)) = u^{-1}\cdot (\dd \kappa (v\circ f_u^{-1}) \,u\cdot \xi) + u^{-1}\cdot (\kappa\circ f_u^{-1}\, \cD_{v}(u\cdot\xi)) =   \dd \kappa (v)\, \xi  + \kappa\, \cD^u_{v}(\xi) \, ,
\end{eqnarray*}
where $\kappa \in C^{\infty}(X)$ is a function on $X$. On the other hand, 
the push-forward  of $\cJ$ by $u\in \Aut(\cS)$ is given by:
\begin{equation*}
\cJ_u(\xi) \eqdef (u\cdot\cJ(u^{-1}\cdot \xi)) = u\circ\cJ(u^{-1}\circ \xi)\, ,  
\end{equation*}
for every $\xi\in \Gamma(\cS)$.
\end{remark}

\noindent
Given a duality bundle $\Delta$ over the scalar bundle $(\pi,\cH,\cG)$, every element $u\in \Aut(\Delta)$ maps a triplet of the form:
\begin{equation*}
(g,s,\cF)\in \Lor(M)\times \Gamma(\pi)\times \Omega^2(M,\cS^{s})\, ,
\end{equation*}
to a triplet of the form:
\begin{equation*}
\mathbb{A}_u(g,s,\cF) \eqdef (g, f_u\circ s, u\cdot\cF)\in \Lor(M)\times \Gamma(\pi)\times \Omega^2(M,\cS^{f_u(s)})\, ,
\end{equation*}
where \emph{dot} denotes the natural action of $\Aut(\cS)$ on $\cS^{s}$-valued forms. 

\begin{remark}
Recall that $\cF_m\in \wedge^2 T^{\ast}_m M\otimes\cS^s_{m}$ or, equivalently:
\begin{equation*}
\cF_m\in \wedge^2 T^{\ast}_m M\otimes\cS_{s(m)}\, .
\end{equation*}
The push-forward $u\cdot\cF\in \Omega^2(M,\cS^{f_u(s)})$ of
$\cF\in \Omega^2(M,\cS^{s})$ by $u$ produces a
$\cS^{f_u(s)}$-valued two-form on $M$ whose value at $m \in M$
is given by:
\begin{equation*}
(u\cdot \cF)_m = u_{s(m)}(\cF_m)\in \wedge^2 T^{\ast}_m M\otimes\cS_{f_u(s(m))}\, ,
\end{equation*} 
where $u_{s(m)}\colon \cS_{s(m)} \to
\cS_{f_u(s(m))}$ acts trivially on the two-form components of
$\cF$.
\end{remark}
Given $u\in\Aut(\Delta)$, the map $\mathbb{A}_u$ defined above need
not preserve the configuration space $\Conf(\Phi)$ defined by a fixed
scalar-electromagnetic bundle $\Phi = (\pi,\cH,\cG,\Xi)$. Instead we
have the following result.

\begin{thm}
\label{thm:equivsolutions}
Let $\pi\colon X\to M$ be a smooth submersion. For every connection
$\cH$, vertical metric $\cG$ and electromagnetic bundle $\Xi$ on
$\pi$, an element $u\in \Aut_{\pi}(\Delta)$ defines a bijection:
\begin{equation*}
\mathbb{A}_u \colon \Conf(\Phi) \xrightarrow{\sim} \Conf(\Phi_u)\, , \quad (g,s,\cF) \mapsto (g, f_u\circ s, u\cdot\cF)\, ,
\end{equation*}
which restricts to a bijection:
\begin{equation*}
\mathbb{A}_u \colon \Sol(\Phi) \xrightarrow{\sim} \Sol(\Phi_u)\, ,
\end{equation*}
between the solution spaces of the bosonic supergravities associated to $\Phi$
and $\Phi_u$ on $(\pi,\cH,\cG)$.
\end{thm}
 
\begin{proof}
Assume that $(g,s,\cF)\in \Conf(\Phi)$, where $\Phi = (\pi,\cH,\cG,\Delta,\cJ)$ 
and $u\in \Aut_{\pi}(\Delta)$ covers $f_u\in \Aut_b(X)$. Clearly,
$f_u\circ s \colon M\to X$ is again a section of $\pi$ since $f_u\colon
X\to X$ is covers the identity over $M$. On the other hand, $u\cdot\cF$
is by construction a two-form on $M$ taking values in
$\cS^{f_u(s)}$ whence $(g,s,\cF)\in \Conf(\Phi_u)$. The fact that this map 
takes solutions to solutions follows by a computation that involves several
different pull-backs through unbased automorphisms of fiber bundles. The reader 
is referred to \cite[Appendix D]{Lazaroiu:2016iav} for a detailed account of 
the operations involved. Assume that $(g,s,\cF)\in\Sol(\Phi)$. For the 
Einstein equation \eqref{eq:GlobalEinstein}, we compute:
\begin{equation}
\label{eq:proof1}
s^{\ast}_{\cC}\cG = (f_u^{-1}\circ f_u\circ s)^{\ast}_{\cC}\cG = (f_u\circ s)^{\ast} ((f_u^{-1})^{\ast}_{\cC}\cG) = (f_u\circ s)^{\ast}_{f_{u\ast}\cC}\cG	\, .
\end{equation}
On the other hand, we have:
\begin{equation}
\label{eq:proof2}
\cF \oslash_{Q^{s}} \cF = (u^{-1}\cdot u\cdot \cF) \oslash_{Q^{s}} (u^{-1}\cdot u\cdot\cF) = (u\cdot \cF) \oslash_{Q^{f_u(s)}_u} ( u\cdot\cF)\, ,
\end{equation}
where $Q^{f_u(s)}_u$ denotes the bilinear form on $\cS^{s}$ defined as follows:
\begin{equation*}
Q^{f_u(s)}_u(\xi^{f_u(s)} , \xi^{f_u(s)}) = \omega ( \xi(f_u(s)) , \cJ(f_u(s)) \xi(f_u(s)) ) \, ,
\end{equation*}
and where $\xi^{f_u(s)} \in \cS^{f_u(s)}$ for every $\xi\in
\cS$. Combining equations \eqref{eq:proof1} and \eqref{eq:proof2}
together with the fact that the left hand side of Equation
\eqref{eq:GlobalEinstein} is invariant under $u$ we obtain that
$(g,f_u\circ s,u\cdot\cF)$ satisfies the Einstein equations with
respect to the scalar-electromagnetic structure
$(\pi,\cG,\cH_u,\Delta_u,\cJ_u)$. For the scalar equation
\eqref{eq:GlobalScalar}, we compute:
\begin{equation*}
\nabla^{\Phi(g,s)}  \dd^{\cC}s =  \nabla^{\Phi(g,s)}  \dd^{\cC}(f_u^{-1}\circ f_u\circ s)  =  \nabla^{\Phi_u(g_u,f_u(s))}  \dd^{f_{u\ast}\cC}(f_u\circ s)\, .
\end{equation*}
Similarly, using the fact that $u\in \Aut(\Delta)$ preserves the flat connection
$\cD$ determined by $\Delta$ together with equation \eqref{eq:proof2}, we obtain:
\begin{equation*}
(\ast \cF , \Psi^{s}\cF)_{g,Q^{s}} = (\ast (u\cdot\cF) , \Psi^{f_u(s)}_u (u\cdot\cF))_{g,Q^{f_u(s)}}\, , 
\end{equation*}
where we have defined:
\begin{equation*}
\Psi^{f_u(s)}_u = (\cD\cJ_u)^{f_u(s)}\, .
\end{equation*}
Hence $(g,f_u\circ s,u\cdot\cF)$ satisfies the scalar equations associated to
the scalar-electromagnetic structure $\Phi_u = (\pi,\cG,\cH_u,\Delta_u,\cJ_u)$. Since the 
flatness condition for $\cF$ is linear in $\cF$ it is enough to verify it on an homogenous
element of the form $\cF = \alpha\otimes \xi^{s}$, where $\alpha \in \Omega^2(M)$ and 
$\xi^s$ is the pull-back by $s$ of a section $\xi\in \Gamma(\cS)$. We compute:
\begin{eqnarray*}
& \dd_{\cD^{f_u(s)}} (u\cdot\cF) = \dd_{\cD^{f_u(s)}}  (\alpha\otimes u\cdot \xi^{s}) = \dd_{\cD^{f_u(s)}} (\alpha\otimes (u\circ\xi\circ f^{-1}_u \circ f_u(s))) = \dd\alpha\otimes (u\cdot \xi^{s}) \\
& + \alpha\otimes \cD^{f_u(s)} (u\circ\xi\circ f^{-1}_u)^{f_u(s)}  = \dd\alpha\otimes (u\cdot \xi^{s}) + \alpha\otimes (\cD (u\circ\xi\circ f^{-1}_u))^{f_u(s)} = \dd\alpha\otimes (u\cdot \xi^{s}) \\
&  + \alpha\otimes (u\circ\cD\xi\circ f^{-1}_u)^{f_u(s)} = \dd\alpha\otimes (u\cdot \xi^{s}) + \alpha\otimes u\cdot\cD^{s}\xi^{s} = u\cdot \dd_{\cD^s}\cF = 0\, ,
\end{eqnarray*}
where we have usted that $u\in \Aut_{\pi}(\Delta)$ preserves the symplectic connection
$\cD$. Whence $u\cdot \cF$ is flat with respect to $\cD^{f_u (s)}$. On the other hand, 
the Maxwell equation is also linear in $\cF$ hence it is enough to verify it on an homogenous
element $\cF = \alpha\otimes \xi^{s}$. We obtain:
\begin{eqnarray*}
& \star_{g,\cJ^{f_u(s)}_{u}}(u\cdot \cF) = \ast_g \alpha \otimes \cJ^{f_u(s)}_{u} (u\cdot \xi^s) = \ast_g \alpha \otimes \cJ^{f_u(s)}_{u} (u\circ \xi\circ f_u^{-1}\circ f_u(s))  \\
& =  \ast_g \alpha \otimes (\cJ_{u} (u\circ \xi\circ f_u^{-1}))^{f_u(s)} =  \ast_g \alpha \otimes (u\circ\cJ(\xi)\circ f_u^{-1}))^{f_u(s)} = \ast_g \alpha \otimes u\cdot \cJ^s (\xi^s) =u\cdot  \star_{g,\cJ^{s}}\cF  =u\cdot \cF \, ,
\end{eqnarray*}
whence $(g,f_u\circ s,u\cdot\cF)$ also satisfies the Maxwell equations 
associated to $(\pi,\cG,\cH_u,\Delta_u,\cJ_u)$. Thus 
$(g,f_u\circ s,u\cdot\cF)\in\Sol(\pi,\cH_u,\cG_u,\Delta_u,\cJ_u)$ and
reversing the previous relations it is easy to see that the map
$(g,s, \cF)\mapsto (g,f_u\circ s,u\cdot\cF)$ is a bijection.
\end{proof}

\begin{remark}
The group $\Aut_{\pi}(\Delta)$ is the global counterpart of the
so-called \emph{pseudo-duality group} introduced in \cite{Hull:1995gk}
as the direct product of the symplectic group and the diffeomorphism
group of the simply connected open set on which the theory is considered. 
When both $\Delta$ and $\pi$ are non-trivial, the group $\Aut_{\pi}(\Delta)$ 
can differ markedly from the local pseudo-duality group of loc. cit.
\end{remark}

\noindent
The following statement follows from \cite[Lemma 4.2.8]{Donaldson}.

\begin{lemma}
\label{lemma:Holfinite}
Let $\Delta$ be a duality bundle over the submersion $\pi:X\rightarrow
M$ and consider a point $x \in X$. Then there exists a canonical
isomorphism:
\begin{equation*}
\Aut_b(\Delta) = \mathrm{C}(\mathrm{Hol}_x(\cD), \Aut(S_x,\omega_x))\, ,
\end{equation*}
where $\mathrm{Hol}_x(\cD)$ is the holonomy group of $\cD$ at $x$,
$\Aut(S_x,\omega_x) \simeq \mathrm{Sp}(2n_v,\mathbb{R})$ is the
automorphism group of the fiber $(S_x, \omega_x) =(S,\omega)\vert_x$
and $\mathrm{C}(\mathrm{Hol}_x(\cD), \Aut(S_x,\omega_x)))$ denotes the
centralizer of $\mathrm{Hol}_x(\cD)$ in $\Aut(S_x,\omega_x)$.
\end{lemma}

\noindent
Fixing $x\in X$, this shows that \eqref{ses1} is isomorphic with the exact sequence:
\begin{equation*}
1 \to \mathrm{C}(\mathrm{Hol}_x(\cD), \Aut(S_x,\omega_x))) \to \Aut_{\pi}(\Delta) \to \Aut^0_b(\pi) \to 1\, .
\end{equation*}

\noindent
Given a scalar bundle $(\pi,\cH,\cG)$ and an electromagnetic bundle
$\Xi = (\Delta,\cJ)$ we next introduce a subgroup of $\Aut_\pi(\Delta)$ 
which preserves the Ehresmann connection $\cH$, the vertical metric $\cG$ 
and the vertical taming $\cJ$. This subgroup gives the global counterpart 
of the group of \emph{continuous} U-dualities studied traditionally in
the supergravity literature.

\begin{definition}
Let $\Phi = (\pi,\cH,\cG,\Delta,\cJ)$ be a scalar-electromagnetic
bundle on $M$. The {\bf classical U-duality group} of $\Phi$ is the
subgroup $\U(\Phi)$ of $\Aut_{\pi}(\Delta)$ consisting of those
elements which preserve the Ehresmann connection $\cH$, the vertical
metric $\cG$ and the vertical taming $\cJ$:
\begin{equation*}
\U(\Phi) \eqdef \left\{ u\in \Aut_{\pi}(\Delta)\,\,\vert\,\,  \cH_u =\cH \, , \, \, \cG_u =\cG \, , \, \, \cJ_{u} = \cJ \right\}\, .
\end{equation*}  
\end{definition}

\noindent
Similarly, we denote by $\U_o(\Phi) \subset \U(\Phi)$ the subgroup of 
$\U(\Phi)$ consisting of those elements that cover diffeomorphisms of $X$ 
isotopic to the identity. Let $\Aut_b(\Xi) \subset \Aut_b(\Delta) $ be the 
group based automorphisms of $\Xi$, which consists of those vector bundle
automorphisms of $\cS$ which cover the identity and preserve $\omega$,
$\cD$ and $\cJ$. The U-duality group fits into a short exact sequence:
\begin{equation*}
1 \to \Aut_b(\Xi)\to \U(\Phi)\to \Aut_b^0(\pi,\cH,\cG) \to 1\, ,
\end{equation*}

\noindent where $\Aut_b^0(\pi,\cH,\cG) \subset \Aut_{b}(\pi)$ denotes
the subgroup of those based automorphisms of $\pi$ that can be covered
by elements of $\U(\Phi)$ and preserve both the Ehresmann connection
$\cH$ and the vertical metric $\cG$. If the scalar bundle
$(\pi,\cH,\cG)$ is flat then the group $\Aut_b(\pi(\cH,\cG))$ is
finite-dimensional by Lemma \ref{lemma:Holfinite}, which in turn
implies that $\U(\Phi)$ is a finite-dimensional Lie group. In general,
this group is markedly different from the U-duality group
traditionally considered in the local formulation of the theory. The
main feature of the latter is that maps solutions to solutions and
hence it can be used as a solution generating mechanism. This key
property also holds for $\U(\Phi)$ as a consequence of Theorem
\ref{thm:equivsolutions}.

\begin{cor}
The action $\mathbb{A}$ of the U-duality group $\U(\Phi)$ preserves
both $\Conf(\Phi)$ and $\Sol(\Phi)$. i.e. it maps configurations to
configurations and solutions to solutions. Moreover, if the scalar
bundle $(\pi, \cH,\cG)\in \Phi$ is flat then $\U(\Phi)$ is a
finite-dimensional Lie group.
\end{cor}
 
\noindent For further reference we introduce the following definition.

\begin{definition} The \emph{classical U-duality transformation}
defined by an element $u\in \U(\Phi)$ is the bijection
$\mathbb{A}_u\colon \Sol(\Phi) \to \Sol(\Phi)$.
\end{definition}


\section{The Dirac-Schwinger-Zwanziger integrality condition}
\label{sec:DQsymplecticabelian}


This section discusses the geometric model obtained by imposing the
DSZ integrality condition on the universal bosonic sector of
four-dimensional supergravity defined by a fixed
scalar-electromagnetic bundle. This condition depends on the choice of
a {\em Dirac system} for the underlying duality bundle $\Delta$ and of
the choice of an \emph{integral} cohomology class in $H^2(M,\Delta)$, 
where integrality is defined relative to that Dirac system.


\subsection{The vector space of integral field strengths}


The DSZ quantization condition of local supergravity is implemented
using a full symplectic lattice. Similarly, we implement the DSZ
quantization of the universal bosonic sector defined by a
scalar-electromagnetic bundle $\Phi$ in terms of a smoothly varying
fiber-wise choice of full symplectic lattices for the underlying
duality bundle $\Delta$, as proposed in \cite{Lazaroiu:2016iav}.
Recall that a full lattice $\Lambda$ in a $2n$-dimensional symplectic
vector space $(V,\omega)$ is called {\em symplectic} if the restriction
of the symplectic pairing $\omega$ to $\Lambda$ takes integer values.
Such lattices are characterized up to symplectomorphism by their {\em type}
$\frt\in \Div^n$ (see \cite[Proposition 1.1]{Debarre}), where:
\be
\Div^n\eqdef \{\frt=(t_1,\ldots, t_n)\in \Z_{>0}^n~\vert~t_1 | t_2 | \ldots | t_n \}~~.
\ee
Any full symplectic lattice of type $\frt\in \Div^n$ in $(V,\omega)$
admits a basis $\lambda_1,\ldots, \lambda_n,\mu_1,\ldots, \mu_n$ such
that:
\be
\omega(\lambda_i,\mu_j)=t_j\delta_{ij}~~,~~\omega(\lambda_i,\lambda_j)=\omega(\mu_i,\mu_j)=0~~\forall i,j=1,\ldots,n~~
\ee
and any element $\frt\in \Div^n$ is realized as the type of some full
symplectic lattice. The symplectic lattice $\Lambda$ is called {\em
principal} if $\frt=\delta_n\eqdef (1,\ldots, 1)$. The {\em modified
Siegel modular group} of type $\frt\in \Div^n$ is the subgroup
$\Sp_\frt(2n,\Z)\subset \Sp(2n,\R)\simeq \Aut(V,\omega)$ consisting
of those symplectic transformations which preserve the standard
lattice of type $\frt$ in $\mathbb{R}^{2n}$. We have $\Sp_{\delta_n}(2n,\Z)=\Sp(2n,\Z)$ 
and $\Sp_\frt(2n,\Z) \subset \Sp(2n,\Z)$ for all $\frt\in \Div^n$.

\begin{definition}
Let $\Delta=(\cS,\omega,\cD)$ be a duality bundle on the scalar bundle
$(\pi,\cH,\cG)$ with submersion $\pi:X\rightarrow M$. A {\em Dirac
system} on $\Delta$ is a smooth fiber sub-bundle $j\colon
\cL\hookrightarrow \cS$ of full symplectic lattices in $(\cS,\omega)$
which is preserved by the parallel transport $T$ of the flat
connection $\cD$ in the sense that we have:
\begin{equation*}
T_{\gamma}(\cL\vert_{\gamma(0)}) = \cL\vert_{\gamma(1)}
\end{equation*} for any piece-wise smooth path $\gamma\in \cP(X)$. The
common type of these fiberwise symplectic lattices is called the {\em
type} of $\cL$. A pair:
\begin{equation*}
\bDelta  \eqdef (\Delta, \cL)\, , 
\end{equation*}
consisting of a duality bundle $\Delta$ and a choice of Dirac system
$\cL$ for $\Delta$ is called an {\em integral duality bundle}.
\end{definition}

\noindent For every $x\in X$, the fiber $(\cS_x, \omega_x,\cL_x)$ of
an integral duality bundle $\bDelta = (\Delta, \cL)$ with
$\Delta=(\cS,\omega,\cD)$ is an \emph{integral symplectic space} as
defined in \cite[Appendix B]{LazaroiuShahbaziAGT}. All fibers of
$\bDelta$ are isomorphic as integral symplectic spaces, hence their
type does not depend on $x\in X$ since we assume that $X$ is
connected.

\begin{remark}
The existence of a Dirac system is obstructed. A duality bundle
$\Delta=(\cS,\omega,\cD)$ of rank $2n_v$ admits a Dirac system of type
$\frt\in \Div^n$ if and only if the structure group of $\cS$ can be
reduced from $\Sp(2n,\R)$ to $\Sp_{\mathfrak{t}}(2n_v,\mathbb{Z})$. We
say that $\Delta$ is \emph{semiclassical} if it admits a Dirac system.
\end{remark}

\begin{definition}
Let $\bDelta_1 = (\Delta_1 ,\cL_1)$ and $\bDelta_2 = (\Delta_2
,\cL_2)$ be two integral duality bundles on $M$. A {\em morphism} of
integral duality bundles from $\bDelta_1$ to $\bDelta_2$ is a morphism
of duality bundles $f\colon \Delta_1 \to \Delta_2$
such that $f(\cL_1) = \cL_2$. 
\end{definition}
 
\begin{remark}
Given a Dirac system $\cL$ for a duality bundle
$\Delta=(\cS,\omega,\cD)$, let $\fS_{\bDelta}\eqdef\cC(\cL)$ be the
locally-constant sheaf of continuous sections of the discrete fiber
bundle $\cL$. This is a subsheaf of the sheaf $\fS_{\bDelta}$ of flat
sections of $(\cS,\cD)$ whose stalk at $x\in X$ identifies with the
symplectic lattice $\cL_x\subset \cS_x$.
\end{remark}

\noindent For every scalar section $s\in \Gamma(\pi)$, let
$\fS^{s}_{\bDelta}\eqdef s^\ast(\fS_{\bDelta})$ be the
locally constant sheaf on $M$ obtained as the pullback of
$\fS_\bDelta$ through $s$.  The sheaf cohomology groups
$H^k(M,\fS^{s}_{\bDelta})$ are naturally isomorphic with the 
cohomology  groups $H^k(M,\cL^s)$ of $M$ with coefficients in 
the local system $\cL^s=s^\ast(\cL)$ and play a crucial role 
in what follows.

\begin{definition}
An \emph{integral
electromagnetic bundle} is a pair:
\begin{equation*}
\bXi \eqdef (\Xi,\cL)\, ,
\end{equation*}
where $\Xi$ is an electromagnetic bundle on $M$ and $\cL$ is Dirac
system for the duality bundle of $\Xi$. An \emph{integral
scalar-electromagnetic bundle} on $M$ is a pair:
\begin{equation*}
\bPhi \eqdef (\Phi,\cL)\, ,
\end{equation*}
where $\Phi$ is a scalar-electromagnetic bundle on $M$ and $\cL$ is a
Dirac system for the duality bundle of $\Phi$.
\end{definition}

\noindent
Given an integral electromagnetic bundle $\bXi=(\Xi,\cL)$ with
integral duality structure $\bDelta=(\cS,\omega,\cD,\cL)$ over a
submersion $\pi:X\rightarrow M$, the quotient:
\begin{equation*}
\cX_{\bDelta} \eqdef \cS/\cL
\end{equation*}
is a flat fibration over $X$ by symplectic torus groups. The taming
$\cJ$ of $\Delta$ makes this into a fibration by polarized Abelian
varieties which however need not be flat since $\cJ$ is not flat
unless the underlying electromagnetic bundle $\Xi$ is unitary.  The
sheaf $\fS_{\cX_\bDelta}$ of smooth flat sections of $\cX_{\bDelta}$ fits
into a short exact sequence of sheaves of Abelian groups defined on
$X$:
\begin{equation*}
0\to \fS_{\bDelta} \xrightarrow{j} \fS_{\Delta} \to \fS_{\cX_{\bDelta}} \to 0\, .
\end{equation*}
which pulls-back to a short exact sequence of sheaves of Abelian
groups defined on $M$:
\begin{equation*} 
0\to \fS^{s}_{\bDelta} \xrightarrow{j^{s}} \fS^{s}_{\Delta} \to \fS^s_{\cX_\bDelta} \to 0\, .
\end{equation*}
The latter induces a long exact sequence in sheaf cohomology, of which
we are interested in the following portion:
\begin{equation*}
\ldots \rightarrow H^1(M,\fS^{s}_{\cX_{\bDelta}})\to H^2(M,\fS^{s}_{\bDelta}) \xrightarrow{j^{s}_{\ast}} H^2(M,\fS^{s}_{\Delta})\to  H^2(M,\fS^s_{\cX_{\bDelta}})\rightarrow \ldots\, .
\end{equation*}

\begin{definition}
The {\em charge lattice} of the integral scalar-electromagnetic
structure $\bXi=(\Xi,\cL)$ relative to the scalar section $s\in
\Gamma(\pi)$ is the lattice:
\be
L_{\bXi}^s\eqdef j^{s}_{\ast}(H^2(M,\fS^{s}_{\bDelta})) \subset H^2(M,\fS^{s}_{\Delta})\, ,
\ee
Elements of this lattice are called \emph{integral cohomology
classes}.
\end{definition}

\noindent It can be shown that $L_{\bXi}^s$ is a full lattice in
$H^2(M,\fS^{s}_{\Delta})$ (see \cite[Proposition
2.24]{LazaroiuShahbaziAGT}). Given an integral scalar-electromagnetic
bundle $\bPhi$, we implement DSZ quantization by restricting the
configuration space $\Conf(\Phi)$ to a subset $\Conf(\bPhi)\subset
\Conf(\Phi)$ obtained by imposing an \emph{integrality condition} on
the elements of $\Conf(\Phi)$. This is the appropriate implementation 
of the DSZ quantization condition in our situation. 

\begin{definition}
Let $\bPhi$ be an integral scalar-electromagnetic bundle. The {\em
integral configuration space} $\Conf(\bPhi)$ of defined by $\bPhi$ is the
set:
\begin{equation*}
\Conf(\bPhi) \eqdef \left\{ (g,s,\cF) \in \Conf(M,\Phi) \,\, \vert \,\, [\cF]\in 2\pi L^s_\Xi\right\}\, .
\end{equation*}
The {\em integral solution space} $\Sol(\bPhi)\subset \Conf(\Phi)$
defined by $\bPhi$ is the set:
\be
\Sol(\bPhi)\eqdef \Sol(\Phi)\cap \Conf(\bPhi)~~.
\ee
\end{definition}

\noindent For further reference we introduce a refinement of the
previous definition.

\begin{definition} Let $\bPhi$ be an integral scalar-electromagnetic
bundle and let $\mathfrak{V}\in H^2(X,\fS_{\bDelta})$. The {\em
framed integral configuration space} $\Conf(\frV,\bPhi)$ with
\emph{framing} $\frV$ of the classical geometric supergravity theory
associated to $\bPhi$ is defined as the following subset of
$\Conf(\bPhi)$:
\begin{equation*}
\Conf(\frV,\bPhi) \eqdef \left\{ (g,s,\cF) \in \Conf(\Phi) \,\, \vert \,\, [\cF] = 2\pi j^{s}_{\ast}(\frV^{s}) \right\}\, ,
\end{equation*}
The {\em framed integral solution space}
$\Sol(\frV,\bPhi)\subset \Conf(\frV,\Phi)$ is the set:
\be
\Sol(\frV,\bPhi)\eqdef \Sol(\Phi)\cap \Conf(\frV,\bPhi)~~.
\ee
\end{definition}

\begin{definition}
The \emph{arithmetic U-duality group} of an integral scalar-electromagnetic
structure $\bPhi=(\Phi,\cL)$ is the subgroup of $\U(\Phi)$ defined through:  
\begin{equation*}
\U(\bPhi) \eqdef \left\{ u\in \U(\Phi) \,\, \vert\,\, u(\cL) =\cL\right\}\, .
\end{equation*}
Similarly $\U_o(\bPhi) \eqdef \left\{ u\in \U_o(\Phi) \,\, \vert\,\, u(\cL) =\cL\right\}$.
\end{definition}

\noindent The arithmetic U-duality group $\U(\bPhi)$ is the global
counterpart of the arithmetic U-duality group of local supergravity
normally considered in the physics literature \cite{Hull:1994ys,Mizoguchi:1999fu}. 
We remark that the supergravity literature seems to have considered thus far only holonomy
trivial Dirac systems $\cL$ of principal type, though there is a priori no physical 
or mathematical reason to make that assumption. We will consider some simple examples
of arithmetic U-duality groups in Section \ref{sec:globalautgroup}. More elaborated 
examples will be considered in a separate publication.


\section{The DSZ quantization of 4d bosonic supergravity}


In this section we describe the geometric and gauge-theoretic formulation 
of the universal bosonic sector of 4d supergravity implied by the DSZ
quantization condition.  This formulation can be constructed through a
step-by-step process as done in \cite{LazaroiuShahbaziAGT} for Abelian gauge
theory. Instead of going through the details of that process, which are similar
to those in \cite{LazaroiuShahbaziAGT}, we give the description of the theory in 
its final form, verifying then that it satisfies the appropriate DSZ quantization
(see Theorem \ref{thm:DQSugra}). The key ingredient occurring in the construction 
is a \emph{Siegel bundle}, a special kind of principal bundle which was introduced 
and discussed in detail in \cite[Section 3]{LazaroiuShahbaziAGT} and forms a 
particular case of the more general notion of principal bundle with weakly-Abelian 
structure group studied in \cite{wa}, to which we refer the reader for background 
and further details. Given $\frt\in \Div^{n_v}$, we define the following disconnected 
Lie group:
\be
\Aff_\frt\eqdef \U(1)^{2n_v}\rtimes \Sp_\frt(2n,\Z)~~,
\ee
where $\U(1)^{2n_v}\simeq \R^{2n_v}/\Z^{2n_v}$ is an affine 
torus group of dimension $2n_v$. The group $\Aff_\frt$
identifies with the set $\U(1)^{2n_v}\times \Sp_\frt(2n_v,\Z)$
equipped with the multiplication rule:
\begin{equation*}
(a_1,\gamma_1)\,(a_2,\gamma_2 ) =(a _1+ \gamma_1 a_2,
\gamma_1\gamma_2)\, , \quad \forall \,\, a_1,a_2\in \U(1)^{2n_v}\, ,
\quad \forall\,\, \gamma_1,\gamma_2\in \Sp_\frt(2n_v,\Z)\, .
\end{equation*}
The modified Siegel modular group $\Sp_\frt(2n_v,\Z)$
coincides with the automorphism group of the standard integral
symplectic space $(\mathbb{R}^{2n_v},\omega_{n_v},\wedge_{\frt})$ of type
$\frt$, where $\omega_{n_v}$ is the standard symplectic form on
$\mathbb{R}^{2n_v}$ and:
\be
\Lambda_{\frt} \eqdef \Z^{n_v}\oplus \oplus_{i=1}^{n_v}{t_i \Z}\subset \mathbb{R}^{2n_v}
\ee
is the standard symplectic lattice of type $\frt$ (see \cite[Appendix
B]{LazaroiuShahbaziAGT}). Moreover, $\Aff_\frt$ coincides with the
group of affine symplectomorphisms of the $2n_v$-dimensional
symplectic torus $(\R^{2n_v}/\Lambda_\frt, \Omega_\frt)$, whose
symplectic form $\Omega_\frt$ is induced by $\omega_{n_v}$. The
connected component of the identity in $\Aff_{\frt}$ is the torus
group $\U(1)^{2n_v}$, while the group of components of $\Aff_\frt$ is
the discrete group $\Sp_\frt(2n_v,\Z)$, which is infinite and
non-Abelian when $n_v>0$.

\begin{definition}
A {\em Siegel bundle} $P_{\frt}$ of rank $n_v$ and type $\frt\in
\mathrm{Div}^{n_v}$ on $X$ is a principal bundle defined on $X$ with
structure group $\Aff_\frt$. A {\em based isomorphism of Siegel
bundles} is a based isomorphism of principal bundles.
\end{definition}

\noindent Let $(\pi, \cH,\cG)$ be a scalar bundle with submersion
$\pi:X\rightarrow M$ and consider a Siegel bundle $P_{\frt}$ of 
rank $n_v$ and type $\frt\in \Div^{n_v}$ over $X$. As shown in
\cite{LazaroiuShahbaziAGT}, the adjoint bundle of $P_{\frt}$ admits a 
natural structure of integral duality bundle of type $\frt$ which we 
denote by $\bDelta(P_{\frt})$. By definition, a vertical taming $\cJ$
of $P_{\frt}$ is a vertical taming of $\bDelta(P_{\frt})$. Given such 
a taming, the  pair $(P_{\frt},\cJ)$ is called a (positively) polarized 
Siegel  bundle (cf. \cite{LazaroiuShahbaziAGT,wa}). The integral electromagnetic
bundle $\bXi(P_{\frt},\cJ)$ determined by $(P_{\frt},\cJ)$ is defined through:
\begin{equation*}
\bXi(P_{\frt},\cJ) \eqdef (\bDelta(P_{\frt}),\cJ)\, .
\end{equation*}
 
\noindent Given a scalar section
$s\in\Gamma(\pi)$, we denote by $P^{s}_{\frt}$ the pullback of
$P_{\frt}$ by $s$, which becomes a Siegel bundle over $M$. Similarly, 
we denote by $\bDelta(P^{s}_{\frt})$ and $\bXi(P^{s}_{\frt},\cJ^{s})$
the integral duality and integral electromagnetic bundles defined by
$P^{s}$ and $\cJ^{s}$, which coincide with the
$s$-pullbacks of the corresponding bundles defined by $P$ and
$\cJ$ on $X$. When necessary, we will write:
\begin{equation*}
\bDelta(P^{s}_{\frt}) = (\cS^{s},\omega^{s},\cD^{s})\, .
\end{equation*}
Let $\Conn(P^{s}_{\frt})$ be the affine space of connections
on $P^{s}_{\frt}$. Elements of this space are invariant one-forms
on $P_{\frt}$ mapping the fundamental vector fields of $P_{\frt}$ to 
their generators in $\aff_\frt$, where $\aff_\frt\simeq \R^{2n_v}$ is the 
Lie algebra of $\Aff_\frt$, which has trivial Lie bracket. The adjoint 
curvature of a connection $\cA\in \Conn(P^{s}_\frt)$ will be denoted 
by $\cF_\cA\in \Omega^2(M,\cS^s)$. This bundle-valued 2-form is
$\dd_{\cD^s}$-closed by the Bianchi identity since (by the results
of \cite{LazaroiuShahbaziAGT,wa}) all connections on
$P^s$ induce the same connection on $\cS^s$, which
coincides with the connection induced by $\cD^s$ on the adjoint
bundle of $P^s_{\frt}$. Thus:
\be
\dd_\cD^s \cF_\cA = 0\, .
\ee

\begin{definition}
A {\em scalar-Siegel bundle} of rank $n_v$ and type $\frt\in
\Div^{n_v}$ over $M$ is a system $\zeta\eqdef (\pi,\cH,\cG, P_\frt)$, 
where $(\pi,\cH,\cG)$ is a scalar bundle over $M$ with submersion
$\pi:X\rightarrow M$ and $P_\frt$ is a Siegel bundle of rank $n_v$ and 
type $\frt\in \Div^{n_v}$ defined on $X$. Given a vertical taming $\cJ$ of
$\Delta(P_\frt)$, the system $\bzeta\eqdef (\Psi,\cJ)$ is called a {\em
polarized scalar-Siegel bundle} of rank $n_v$ and type $\frt$ over
$M$.
\end{definition}

\begin{definition}
\label{def:confdqsugra}
Let $\bzeta\eqdef (\pi,\cH,\cG, P_{\frt},\cJ)$ be a polarized scalar-Siegel
bundle over $M$. The {\em configuration space} of the bosonic supergravity 
defined by $\bzeta$ is the set:
\begin{equation*}
\cConf(\bzeta) = \left\{ (g,s,\cA)  \,\, \vert \,\, g\in \Lor(M) \, , \,\, s\in \Gamma(\pi)\, , \,\, \cA  \in \Conn(P^{s}_\frt)   \right\}\, .
\end{equation*}
The {\em universal bosonic sector} of four-dimensional supergravity
determined on $M$ by $\bzeta$ is defined through the following system
of partial differential equations for triples $(g,s,\cA)\in
\cConf(\bzeta)$:
\begin{itemize}
\item The Einstein equations:
\begin{equation}
\label{eq:GlobalEinsteinDQ}
\mathrm{Ric}^g - \frac{g}{2} \mathrm{R}^g = \frac{1}{2} \mathrm{Tr}_g(s^{\ast}_{\cC}\cG)\, g - s^{\ast}_{\cC}\cG   + 2 \cF_{\cA} \oslash_{Q^{s}} \cF_{\cA}\, .
\end{equation}
\item The scalar equations:
\begin{equation}
\label{eq:GlobalScalarDQ}
\mathrm{Tr}_g(\nabla^{\Phi(g,s)}  \dd^{\cC}s) = \frac{1}{2} (\ast \cF_{\cA} , (\Psi^s)^{\sharp_{\cG}}\cF_{\cA})_{g,Q^{s}}\, .
\end{equation}
\item The Maxwell equations:
\begin{equation}
\label{eq:GlobalMaxwellDQ}
\star_{g,\cJ^{s}}\cF_{\cA} = \cF_{\cA}\, ,
\end{equation}
whose set of solutions we denote by $\cSol(\bzeta)\subset \cConf(\bzeta)$.
\end{itemize}
\end{definition}

\noindent Connections $\cA$ satisfying equation \eqref{eq:GlobalMaxwellDQ} will
be called {\em polarized self-dual}, following the terminology introduced in
\cite{LazaroiuShahbaziAGT} in the context of Abelian gauge theory.

\begin{remark}
The Maxwell equations of the bosonic gauge sector of local
supergravity are given by a system of second-order partial
differential equations for a number $n_v$ of \emph{electromagnetic}
local gauge potentials whose curvatures satisfy a generalization of 
the Maxwell equations. This is locally equivalent with the description 
given by the first order global equation \eqref{eq:GlobalMaxwellDQ}, 
which reduces locally to a system of first-order partial differential 
equations for $2n_v$ local gauge fields, both \emph{electric} and 
\emph{magnetic} (considered up to gauge transformations of the 
principal bundle $P^s_\frt$).
\end{remark}

\noindent
The Bianchi identity and polarized self-duality condition imply
that the gauge potential of any solution $(g,s,\cA)\in
\Sol(\bzeta)$ automatically satisfies the following second order
differential equation of Yang-Mills type:
\begin{equation*}
\dd_{\cD^s} \star_{g,\cJ}\cF_{\cA} = 0\, .
\end{equation*}
These differ from the usual Yang-Mills equations since $\cF_\cA$
involves both electric and magnetic degrees of freedom while the
equations themselves involve the pulled-back taming $\cJ^s$.

\begin{thm}
\label{thm:DQSugra}
Let $\bPhi = (\pi,\cH,\cG,\bDelta,\cJ)$ be an integral 
scalar-electromagnetic bundle of type $\frt$. For every framed integral 
configuration space $\Conf(\frV,\Phi)$ there exists a vertically polarized 
Siegel bundle $(P_\frt,\cJ)$ on $(\pi,\cH,\cG)$ such that $\bDelta = \bDelta(P_\frt)$
and the twisted Chern class\footnote{See \cite{LazaroiuShahbaziAGT,wa} for 
its precise definition.} $c(P_\frt)$ of $P_{\frt}$ satisfies $c(P_\frt)=\frV$. 
Moreover, the map:
\begin{equation*}
\cSol(\pi,\cH,\cG,P_\frt,\cJ)\to \Sol(\frV,\bPhi) \, , \qquad (g,s,\cA) \mapsto  (g,s,\cF_{\cA})\, ,
\end{equation*}
is surjective.
\end{thm}

\begin{proof}
Given $\bDelta$ and $\frV$, it follows from the results of \cite{BaragliaI,BaragliaII} 
(see also \cite{wa}) that there exists a Siegel bundle $P_{\frt}$ of type 
$\frt$  (unique up to isomorphism) whose twisted Chern class $c(P_{\frt})$ 
equals $\frV$ and whose adjoint bundle is isomorphic to $\bDelta$ as an integral 
duality bundle. The vertical taming $\cJ$ in $\bPhi$ makes $P_{\frt}$ into a 
polarized  Siegel bundle. On the other hand, the curvature of any connection 
$\cA\in \Conn(P^s_{\frt})$ defines a $\dd_{\cD^s}$-cohomology class
$[\cF_{\cA}]_{\cD^s}$ in $H^2(M,\fS_{\Delta^s})$ since, as remarked earlier:
\begin{equation*}
\dd_{\cD^{s}}\cF_\cA = \dd_{\cA}\cF_{\cA} = 0\, .
\end{equation*}

\noindent
Given any other connection $\cA^{\prime}$ on $P_{\frt}$ we have:
\begin{equation*}
\cA^{\prime} = \cA + \bar{\tau}\, ,
\end{equation*}

\noindent
for a unique horizontal and invariant one-form $\bar{\tau} \in \Omega^1(P_{\frt}, 
\mathfrak{aff}_{\frt})$. Therefore, the curvatures of $\cF_{\cA}$ and 
$\cF_{\cA^{\prime}}$ are related as follows:
\begin{equation*}
\cF_{\cA^{\prime}} = \cF_{\cA} + \dd_{\cD^{s}}\tau \in \Omega^2(M,\cS^{s})\, ,
\end{equation*}

\noindent
where $\tau \in \Omega^1(M,\cS^{s})$ is uniquely determined by 
$\bar{\tau} \in \Omega^1(P_{\frt}, \mathfrak{aff}_{\frt})$. This implies that 
the cohomology class $[\cF_{\cA}]_{\cD^s}\in H^2(M,\fS_{\Delta}^s)$ 
does not depend on the connection $\cA$. A similar argument shows that the cohomology 
class $[\cF_{\cA}]_{\cD^s}$ is invariant under automorphisms of $P_{\frt}$ and
therefore only depends on the isomorphism class of the latter. This is further
elaborated in \cite{wa} to show that $[\cF_{\cA}]_{\cD^s}$ is equal to the
\emph{real} twisted Chern class of $P_{\frt}$ as follows:
\begin{equation*}
[\cF_{\cA}]_{\cD^s} = 2\pi j^s_{\ast}(c(P^s_{\frt})) \in L^s_\bDelta\, .
\end{equation*}

\noindent
Since $c(P_{\frt}) = \frV$ by construction, we immediately conclude that:
\begin{equation*}
[\cF_{\cA}]_{\cD^s} = 2\pi j^s_{\ast}(\frV^{s}) \in L^s_\bDelta\, .
\end{equation*}

\noindent
Hence, $(g,s,\cF_{\cA})$ belongs to $\Sol(\frV,\bPhi)$ for all
$(g,s,\cA)\in \Conf(\pi,\cH,\cG,P_{\frt},\cJ)$. Furthermore, every
element in $\Sol(\frV,\bPhi)$ is of the form $(g,s,\cF_{\cA})$ for some
$(g,s,\cA)\in \cSol(\pi,\cH,\cG,P_\frt,\cJ)$. An explicit way to prove 
this is to use a good open cover $M\subset \left\{ U_a \right\}_{a\in I}$ of
$M$. Then, given $(g,s,\cF)\in\Sol(\frV,\bPhi)$, the restriction:
\begin{equation*}
\cF_a \eqdef \cF\vert_{U_a} = \dd \cA_a\, , \qquad \cA_a \in \Omega^1(U_a,\mathbb{R}^{2n_v})\qquad a \in I\, , 
\end{equation*}

\noindent
is $\dd_{\cD^{s}}$-exact and hence $\dd$-exact, since we can
trivialize $\Delta$ over $U_a$ as the latter is simply connected. 
The family of one-forms $\left\{\cA_a \right\}$ taking values 
in $\mathbb{R}^{2n_v}$ can be shown to define a connection on 
$P^s_{\frt}$ whose curvature is precisely $\cF$ and hence
we conclude.
\end{proof}

\noindent The previous theorem shows that Definition
\ref{def:confdqsugra} realizes geometrically the DSZ quantization of
the universal bosonic supergravity sector defined by $\Phi$ since it 
shows that, given a Dirac system for the duality structure of $\Phi$, every
element in the solution space $\Sol(\Phi,\cL)$ can be realized through
a Lorentz metric on $M$, a section of $\pi$ and a gauge potential
$\cA\in \Conn(P^s_{\frt})$ for some Siegel bundle $P^s_{\frt}$ on $X$. The 
latter is the novel geometric object attached to the DSZ quantization condition.


\section{The electromagnetic U-duality group}
\label{sec:globalautgroup}


In this section we investigate the gauge U-duality group of the DSZ
quantization of bosonic supergravity, which yields a natural extension
of its arithmetic U-duality group and provides the geometric interpretation 
of U-duality transformations as gauge transformations.

Fix a vertically polarized Siegel bundle $(P_{\frt},\cJ)$ on the total 
space $X$ of a scalar bundle $(\pi,\cH,\cG)$ and let $\Aut(P_{\frt})$ be 
the automorphism group of $P_{\frt}$. For every $u\in \Aut(P_{\frt})$, 
denote by $\frad_u\colon \bDelta(P_{\frt})\to \bDelta(P_{\frt})$ the 
automorphism of the integral duality structure $\bDelta(P_{\frt})$ defined 
by $u$. Let $\Aut_{\pi}(P_{\frt})\subset \Aut(P_{\frt})$ be the subgroup 
formed by all elements of $\Aut(P_{\frt})$ which cover based automorphisms 
of the fiber bundle $\pi$, that is:
\begin{equation*}
\Aut_{\pi}(P_{\frt}) \eqdef\left\{ u\in \Aut(P_{\frt})\,\, \vert\,\, f_u\in \Aut_b(\pi)\right\}=\left\{ u\in \Aut(P_{\frt})\,\, \vert\,\, \pi\circ f_u=\pi\right\}\, .
\end{equation*}
We have the a short exact sequence of groups:
\begin{equation*}
1 \to \Aut_b(P_{\frt}) \to \Aut_{\pi}(P_{\frt}) \to \Aut^0_{b}(\pi)\to 1\, ,
\end{equation*}
where $\Aut^0_{b}(\pi)$ is the subgroup of $\Aut_{b}(\pi)$ formed by
those based automorphisms of $\pi$ that can be covered by elements of
$\Aut(P_{\frt})$.

\begin{definition}
\label{def:Uduality}
Let $(P_{\frt},\cJ)$ be a vertically polarized Siegel bundle over 
the scalar-bundle $(\pi,\cH,\cG)$. The {\bf gauge U-duality group}
$\U(\bzeta)$ of the polarized scalar-Siegel bundle
$\bzeta=(\pi,\cH,\cG,P_{\frt},\cJ)$ is the subgroup of $\Aut_{\pi}(P_{\frt})$
consisting of those elements which preserve the Ehresmann connection
$\cH$, the metric $\cG$ and the vertical taming $\cJ$:
\begin{equation*}
\U(\bzeta) \eqdef \left\{ u\in \Aut_{\pi}(P_{\frt})\,\,\vert\,\, \cH_{u}=\cH \, , \, \, \cG_{u}=\cG \, , \, \, \cJ_{u} = \cJ \right\}\, .
\end{equation*} 
\end{definition}

\noindent
Similarly, we denote by $\U_o(\bzeta)\subset \U(\bzeta)$ the subgroup of elements that 
cover automorphisms of $\pi$ isotopic to the identity. We will also refer to $\U_o(\bzeta)$
as the gauge U-duality group. Let $\Aut_b(P_{\frt},\cJ)$ be the subgroup of 
$\Aut_b(P_{\frt})$ consisting of those based automorphisms of $P_{\frt}$ which 
preserve $\cJ$. The gauge U-duality group fits into a short exact sequence:
\begin{equation*}
1 \to \Aut_b(P_{\frt},\cJ)\to \U(\bzeta)\to \Aut_b^0(\pi,\cH,\cG) \to 1\, ,
\end{equation*}
where $\Aut^0_b(\pi,\cH,\cG) \subset \Aut^0_b(\pi)$ is the subgroup
consisting of those based automorphisms of $\pi$ that can be covered
by elements of $\U(\bzeta)$ and preserve the Ehresmann 
connection $\cH$ and the metric $\cG$.

The main feature of the local U-duality group of a local supergravity
theory is that maps solutions to solutions and thus can be used as a
solution generating mechanism. This key property also holds for
$\U(\bzeta)$ as we show below. Let:
\begin{equation*}
(g,s,\cA)\in \cSol(\bzeta)~~.
\end{equation*}
Recall that $\cA\in \Conn(P^{s}_{\frt})$ is a connection on the 
pull-back of $P_{\frt}$ through $s\in \Gamma(\pi)$. An element 
$u\in  \Aut_{\pi}(P_{\frt})$ which covers $f_u\in\Aut_b(\pi)$ acts on 
$(g,s,\cA)$ through:
\begin{equation*}
u\cdot (g,s,\cA) = (g,f_u(s),  \cA_{u})\, , 
\end{equation*}
where $f_u(s)=f_u\circ s\in \Gamma(\pi)$ and $\cA_{u}$ is
the push-forward of $\cA$ by the based isomorphism of 
$P^s_{\frt}$ naturally associated to $u$ as follows:
\begin{equation*}
u_m \colon (P^{s}_{\frt})_m = (P_{\frt})_{s(m)} \to (P^{f_u(s)}_{\frt})_m = (P_{\frt})_{f_u(s(m))}\, , \quad
p\mapsto u_{s(m)}(p)\,
\end{equation*}
for all $m\in M$. Notice that $\cA_{u}$ is a connection on the bundle
$P^{f_u(s)}_{\frt}$, where the latter denotes the pull-back of $P_{\frt}$
by the section $f_u(s)\in \Gamma(s)$. Denote by:
\be
\bzeta_u=(\pi,\cH_u,\cG_u,P_{\frt},\cJ_u)
\ee
the push-forward of $\bzeta_u$ of $\bzeta$ by $u \in \Aut_{\pi}(P_{\frt})$.

\begin{cor}
Let $\bzeta=(\pi,\cH,\cG,P_{\frt},\cJ)$ be a polarized scalar-Siegel bundle
with submersion $\pi\colon X\to M$. Every element $u\in \Aut(P)$
defines a bijection of sets:
\begin{equation*}
\mathbb{A}_u \colon \cConf(\bzeta) \to \cConf(\bzeta_u)\, , \quad (g,s,\cA) \mapsto (g, f_u(s), \cA_{u})\, ,
\end{equation*}
which restricts to a bijection:
\begin{equation*}
\mathbb{A}_u \colon \cSol(\bzeta) \to \cSol(\bzeta_u)~~.
\end{equation*}
In particular, if $u\in \U(\bzeta)$ then $\mathbb{A}_u \colon \cSol(\bzeta) 
\to \cSol(\bzeta)$ preserves the solution space of the given supergravity theory.
\end{cor}

\begin{proof}
The result follows directly from Theorem \ref{thm:equivsolutions} upon 
noticing that:
\begin{equation*}
 \cF_{\cA_{u}} = u\cdot\cF_{\cA}\,~~,
\end{equation*}
which shows that $\cF_{A}$ transforms as the field strength $\cF$
considered in Section \ref{sec:geometricsugra} in the classical formulation 
of the theory.
\end{proof}

\noindent For further reference we introduce the following definition.
\begin{definition}
The gauge \emph{U-duality transformation} induced by $u\in
\U(\bzeta)$ is the bijection $\mathbb{A}_u\colon \Sol(\bzeta) 
\to \Sol(\bzeta)$.
\end{definition}

\noindent We have a canonical morphism of groups:
\begin{equation*}
\frad \colon \U(\bzeta) \to \U(\bPhi(\bzeta))\, , \quad u\mapsto \frad_u\, ,
\end{equation*}
where $\bPhi(\bzeta)$ is the integral scalar-electromagnetic bundle
determined by the polarized scalar-Siegel bundle $\bzeta$. This morphism
associates to $u$ the automorphism of the adjoint bundle of $P_{\frt}$ defined 
canonically by the latter.  

\begin{definition}
The \emph{continuous subgroup} of the gauge
U-duality group $\U(\bzeta)$ is:
\begin{equation*}
\mC(\bzeta) \eqdef \ker(\frad)\subset \U(\bzeta)\, .
\end{equation*}
Similarly, $\mC_o(\bzeta) \eqdef \ker(\frad\vert_{\U_o(\bzeta)}) \subset \U_o(\bzeta)$.
\end{definition}

\noindent The classical U-duality group was shown to be a
finite-dimensional Lie group in Section
\ref{sec:classicalglobalautgroup} when the scalar bundle is flat. This
is no longer true for the gauge U-duality group. Instead, if the rank 
of $P_{\frt}\in\U(\bzeta)$ is positive $\U(\bzeta)$ is an extension of 
the arithmetic duality group $\U(\bPhi(\bzeta))$ by an 
\emph{infinite-dimensional} abelian group, a fact that allows us to 
pinpoint the geometric origin of U-duality.

\begin{prop}
\label{prop:extensionUduality}
The gauge U-duality group $\U_o(\bzeta)$ fits into a short exact sequence:
\ben
\label{Cseq}
1 \to \mC_o (\bzeta)  \hookrightarrow \U_o(\bzeta)\xrightarrow{\frad} \U_o(\bPhi(\bzeta))\to 1\, ,
\een
where $\bPhi(\bzeta)$ is the polarized integral scalar-electromagnetic
bundle determined by $\bzeta$.
\end{prop}
 
\begin{proof}
Since it is clear that the natural inclusion $\mC(\bzeta)  \hookrightarrow \U(\bzeta)$ is
injective and the map $\frad$ is a homomorphism, it suffices to prove that $\frad$ is 
surjective. Write the Dirac system $\cL$ in $\bzeta$ as an associated bundle 
$\cL = P_{\frt}\times_\ell \mathbb{Z}^{2n}$ to $P_{\frt}$ through the natural representation 
$\ell$ of $\Aff_\frt$ on $\mathbb{Z}^{2n}$. Let $\phi\in \U(\bPhi(\bzeta))$ be an automorphism 
of the scalar-electromagnetic bundle $\bPhi(\bzeta)$ associated to $\bzeta$ and covering the 
diffeomorphism $f_{\phi}\in \Diff(M)$. Since the latter is a diffeomorphism isotopic to the identity, 
the principal bundles $P_{\frt}$ and $P_{\frt}^{f_{\phi}}$, where the latter denotes the pull-back of 
$P_{\frt}$ by $f_{\phi}$, are isomorphic. Fix such an isomorphism, which is equivalent to 
fixing an automorphism $u^{\prime}_{\phi}\colon P_{\frt}\to P_{\frt}$ covering $f_{\phi}$. Then, 
for every $[p,v] \in \cL$ there exists a unique map $\cB_p\in \mathbb{Z}^{2n} \to \mathbb{Z}^{2n}$ 
such that:
\begin{equation*}
\phi([p,v]) = [u^{\prime}_{\phi}(p), \cB_p(v)]\, . 
\end{equation*}
Since $\phi$ is a linear automorphism of the integral duality structure determined by $\bzeta$
it follows that the map $\cB_p\in \Z^{2n} \to \Z^{2n}$ is a linear automorphism of the standard 
symplectic lattice of type $\frt\in \Z$ and therefore belongs to $\Aff_{\frt}$. Furthermore, 
independence of the representative in $[p,v]\in \cL$ in the definition of $\cB_p$ implies:
\begin{equation*}
\cB_{px} = x^{-1} \circ \cB_p \circ x\, ,
\end{equation*}
for every $x\in \Aff_{\frt}$. Hence, the assignment $p\mapsto \cB_p$ defines a smooth
map $\cB\colon P_{\frt} \to \Aff_{\frt}$ which is equivariant with respect to the 
adjoint action. Therefore, we have:
\begin{equation*}
\phi([p,v]) = [u^{\prime}_{\phi}(p)\cB_p , (v)]\, ,
\end{equation*}
and the automorphism $u_{\phi} \colon P_{\frt}\to P_{\frt}$ defined as follows:
\begin{equation*}
u_{\phi}(p) \eqdef u^{\prime}_{\phi}(p)\cB_p\, , \quad p\in P_{\frt}\, ,
\end{equation*}
covers $f_{\phi}\in \Diff(M)$ and satisfies $\frad(u_{\phi}) = \phi$ by
construction. Hence $\frad$ is surjective and thus we conclude.
\end{proof}

\noindent It is clear that $\frad_u$ is trivial when $u\in
\mC(\bzeta)$. Intuitively speaking, elements in $\mC(\bzeta)$
behave as gauge transformations on a principal torus bundle and
therefore act trivially on the curvature of any connection. In fact,
the arithmetic U-duality group $\U(\bPhi(\bzeta))$ identifies with
the \emph{discrete remnant} (in the sense of \cite{wa}) of the gauge 
group $\Aut(P_{\frt})$, which shows that U-dualities in supergravity
are but gauge transformations of the underlying Siegel bundle, a fact 
that elucidates their geometric origin. We discuss next a few examples. 
An in-depth study of the gauge U-duality group will be presented in a 
separate publication.


\subsection{Rank-zero Siegel bundle}


Let $(\pi,\cH,\cG)$ be a scalar bundle over $M$ with submersion
$\pi:X\rightarrow M$ and consider the rank zero Siegel bundle
$P_0=(\id_X:X\rightarrow X)$ on $X$ (which is necessarily trivial).
In this case $\Aff_{\frt}$ is the trivial group and we have:
\begin{equation*}
\Aut(P_0) = \Diff(X)\, , \qquad \Aut_{\pi}(P_0) = \Aut_b(\pi)\, , \qquad \Aut_{b}(P_0) =\{\id_X\}~~,
\end{equation*} 
as well as:
\be
\bDelta(P_0) = X\times \left\{ 0\right\}\, .
\ee
Let $\bzeta_0=(\pi,\cH,\cG,P_0,\cJ_0)$, where $\cJ_0\eqdef \id_{\bDelta(P_0)}$. 
Then: 
\begin{equation*}
\U(\bzeta_0) = \Aut^0_b(\pi,\cH,\cG) = \left\{ u\in \Aut_{b}(\pi)\,\,\vert\,\, \cH_{u}=\cH \, , \, \, \cG_{u}=\cG \right\}\, .
\end{equation*} 
Lemma \ref{lemma:Holfinite} shows that $\U(\bzeta_0)$ is isomorphic to
the commmutant of the holonomy group of $\cH$ inside the isometry group 
of the typical fiber of $(\cM,\cG)$ of $(\pi,\cG)$. When the holonomy of $\cH$
is trivial, $\U(\bzeta_0)$ reduces to the orientation-preserving isometry group 
$\Iso(\cM,\cG)$ of the scalar manifold but is in general different. In particular, 
when the holonomy of $\cH$ is full, that is, equal to $\Iso(\cM,\cG)$, then 
$\U(\bzeta_0)$ is isomorphic to the center of $\Iso(\cM,\cG)$ and hence possibly
trivial. This gives a simple and explicit example illustrating how the supergravity 
duality group may differ from  its local counterpart considered in the literature, 
which in this case would correspond always with $\Iso(\cM,\cG)$. 


\subsection{Rank zero scalar bundle}


Let $(\pi,\cH,\cG)$ be the rank zero scalar bundle, that is, $X = M$, $\pi 
\colon M\to M$ is the identity map, $\cH = TM$ is canonically identified with the tangent
bundle of $M$ and $\cG$ is the trivial metric on the rank zero vector bundle
over $M$. Then, the isometry group of the typical fiber of $(\pi,\cH,\cG)$ is
the trivial group whence the short exact sequence: 
\begin{equation*}
1 \to \Aut_b(P_{\frt},\cJ)\to \U(\bzeta)\to \Aut_b^0(\pi,\cH,\cG) \to 1\, ,
\end{equation*}
reduces to an isomorphism of groups $\U(\bzeta) = \Aut_b(P_{\frt},\cJ)$ where
$(P_{\frt},\cJ)$ is a polarized Siegel bundle over $M$. Therefore, the gauge U-duality
group reduces to the gauge group of the Siegel bundle underlying the given bosonic
supergravity. This corresponds, in fact, with the electromagnetic gauge duality
group of the abelian gauge theory determined by $(P_{\frt},\cJ)$ as explained in 
detail in \cite{LazaroiuShahbaziAGT}.


\subsection{Holonomy-trivial scalar bundle}


Consider a holonomy-trivial scalar bundle $(\pi,\cH,\cG)$ in the presentation:
\begin{equation*}
X = M \times \cM \, , \qquad \cH=TM^{\mathrm{pr_1}}\, ,
\end{equation*}
where $\cM$ is an oriented $n_s$-dimensional manifold and
$\mathrm{pr}_1\colon M\times \cM\to M$ is the canonical projection
onto the first factor. In this situation $\cV = T\cM^{\mathrm{pr_2}}$,
where $\mathrm{pr}_2\colon M\times \cM\to \cM$ is the canonical
projection onto the second factor, and the vertical metric $\cG$
descends to a Riemannian metric on $\cM$ which we denote by the same
symbol for ease of notation. Furthermore, consider the vertically
polarized Siegel bundle $(P_{\frt},\cJ)$ obtained by pull-back through
$\mathrm{pr}_2\colon M\times \cM\to \cM$ of a vertically polarized
Siegel bundle on $(\cM,\cG)$, which we denote again by
$(P_{\frt},\cJ)$ for ease of notation. Then:
\begin{equation*}
\Aut_b(\pi) = \Diff(\cM)\, , 
\end{equation*}
where $\Diff(\cM)$ the group of oriented diffeomorphisms of $\cM$. In
particular, we obtain the following short exact sequence:
\begin{equation*}
1 \to \Aut_b(P_{\frt}) \to \Aut(P_{\frt}) \to \Diff_0(\cM)\to 1\, ,
\end{equation*}
where $ \Diff_0(\cM)$ denotes the subgroup of $\Diff(\cM)$ that can be
covered by elements in $\Aut(P_{\frt})$. Here $P_{\frt}$ is considered
as a Siegel bundle over $\cM$. In this case, the gauge U-duality group 
is:
\begin{equation*}
\U(\bzeta) =\left\{ u\in \Aut(P_{\frt})\,\,\vert\,\,   \cG_{u} =\cG \, , \, \, \cJ_{u} = \cJ \right\}\, ,
\end{equation*} 
and fits into the short exact sequence:
\begin{equation*}
1 \to \Aut_b(P_{\frt},\cJ) \to \U(\bzeta) \to \Iso_0(\cM,\cG)\to 1\, ,
\end{equation*}
where $\Iso_0(\cM,\cG)$ is group of those isometries of $(\cM,\cG)$
which can be covered by elements in $\U(\bzeta)$. Assume 
in addition that $P_{\frt}$ is topologically trivial and write:
\begin{equation*}
P_{\frt} = \cM\times \Aff_\frt = \cM\times \left[\U(1)^{2n_v}\rtimes \Sp_{\frt}(2n_v,\mathbb{Z})\right]\, .
\end{equation*}
in a fixed trivialization. We have:
\begin{eqnarray*}
& \Aut_b(P_{\frt}) = \cC^{\infty}(\cM ,  \U(1)^{2n_v}\rtimes \Sp_{\frt}(2n_v,\mathbb{Z}))\, , \\
& \Aut(P_{\frt}) = \Diff(\cM)\ltimes \cC^{\infty}(\cM , \U(1)^{2n_v}\rtimes \Sp_{\frt}(2n_v,\mathbb{Z}))\, .
\end{eqnarray*}
Since $\Sp_{\frt}(2n,\mathbb{Z})$ is discrete, we find:
\begin{equation*}
\cC^{\infty}(\cM , \U(1)^{2n_v})\rtimes \Sp_{\frt}(2n_v,\mathbb{Z})) = \cC^{\infty}(\cM , \U(1)^{2n_v})\rtimes \Sp_{\frt}(2n_v,\mathbb{Z})
\end{equation*}
as well as:
\begin{equation*}
\U(\bzeta) \eqdef \left\{ (f,u_T,\mathfrak{U})\in \Iso(\cM,\cG)\ltimes (\cC^{\infty}(\cM , \U(1)^{2n_v})\rtimes \Sp_{\frt}(2n_v,\mathbb{Z}))\,\,\vert\,\, \mathfrak{U} \cJ \mathfrak{U}^{-1} = \cJ\circ f \right\}\, .
\end{equation*} 
In particular, the morphism $\frad \colon \U(\bzeta) \to \U(\bPhi(\bzeta))$ is given explicitly by:
\begin{equation*}
\frad((f,u_T,\mathfrak{U})) = (f,\mathfrak{U})\in \Iso(\cM,\cG)\ltimes \Sp_{\frt}(2n_v,\mathbb{Z}) \, ,
\end{equation*}
The short exact sequence:
\begin{equation*}
1 \to \cC^{\infty}(\cM , \U(1)^{2n_v})  \hookrightarrow \U(\bzeta)\xrightarrow{\frad} \U(\bPhi(\bzeta))\to 1\, ,
\end{equation*}
shows how $\mC(\bzeta) = \cC^{\infty}(\cM , \U(1)^{2n_v})$
captures the \emph{non-discrete} gauge transformations in
$\U(\bzeta)$, which act trivially on the adjoint
bundle of $P$. Consequently, we have:
\begin{equation*}
\U(\bPhi(\bzeta)) \eqdef \left\{ (f,\mathfrak{U})\in \Iso(\cM,\cG)\times \Sp_{\frt}(2n_v,\mathbb{Z})\,\,\vert\,\, \mathfrak{U} \cJ \mathfrak{U}^{-1} = \cJ\circ f \right\}\, ,
\end{equation*} 
which illustrates the explicit dependence of the arithmetic U-duality
group on the type $\frt\in \Div^n$ of its underlying Siegel bundle $P_{\frt}$.


\appendix



\end{document}